\documentclass[12pt,english]{amsart}
\usepackage{amsmath,amsthm,amssymb,amsfonts,amscd, array,latexsym}
\usepackage[noadjust]{cite}
\usepackage[T2A]{fontenc}
\usepackage[cp1251]{inputenc}
\usepackage[english]{babel}
\usepackage{diagmac2} 
\usepackage{tikz}
\usetikzlibrary{positioning,calc,arrows,arrows.meta}

\usepackage{color} 

\newtheorem{theorem}{Theorem}[section]
\newtheorem{corollary}[theorem]{Corollary}
\newtheorem{lemma}[theorem]{Lemma}
\newtheorem{proposition}[theorem]{Proposition}

\theoremstyle{definition}
\newtheorem{definition}[theorem]{Definition}
\newtheorem{example}[theorem]{Example}
\newtheorem{problem}[theorem]{Problem}

\theoremstyle{remark}
\newtheorem{remark}[theorem]{Remark}
\numberwithin{equation}{section}

\DeclareMathOperator{\dom}{dom}

\newcommand{\RR}{\mathbb{R}}

\def\<#1>{\langle #1 \rangle}

\newbox\onebox
\newcommand{\coherent}[1]{\mathbin{\setbox\onebox=\hbox{$=$}\lower0.7\ht%
\onebox\hbox{$\stackrel{#1}{=}$}}}
\newcommand{\acr}{\newline\indent}

\begin{document}

\title[Combinatorial characterization]{Combinatorial characterization of pseudometrics}

\author{Oleksiy Dovgoshey}
\address{\textbf{Oleksiy Dovgoshey}\acr
Function theory department\acr
Institute of Applied Mathematics and Mechanics of NASU\acr
Dobrovolskogo str. 1, Slovyansk 84100, Ukraine}
\email{oleksiy.dovgoshey@gmail.com}

\author{Jouni Luukkainen}
\address{\textbf{Jouni Luukkainen}\acr
Department of Mathematics and Statistics\acr
University of Helsinki\acr
P.O. Box 68 (Pietari Kalmin katu 5)\acr
FI-00014 University of Helsinki, Finland}
\email{jouni.luukkainen@helsinki.fi}

\subjclass[2010]{Primary 54E35, Secondary 20M05.}
\keywords{pseudometric, strongly rigid metric, equivalence relation, semigroup of binary relations.}

\begin{abstract}
Let \(X\), \(Y\) be sets and let \(\Phi\), \(\Psi\) be mappings with the domains \(X^{2}\) and \(Y^{2}\) respectively. We say that \(\Phi\) is \emph{combinatorially similar} to \(\Psi\) if there are bijections \(f \colon \Phi(X^2) \to \Psi(Y^{2})\) and \(g \colon Y \to X\) such that \(\Psi(x, y) = f(\Phi(g(x), g(y)))\) for all \(x\), \(y \in Y\). It is shown that the semigroups of binary relations generated by sets \(\{\Phi^{-1}(a) \colon a \in \Phi(X^{2})\}\) and \(\{\Psi^{-1}(b) \colon b \in \Psi(Y^{2})\}\) are isomorphic for combinatorially similar \(\Phi\) and \(\Psi\). The necessary and sufficient conditions under which a given mapping is combinatorially similar to a pseudometric, or strongly rigid pseudometric, or discrete pseudometric are found. The algebraic structure of semigroups generated by \(\{d^{-1}(r) \colon r \in d(X^{2})\}\) is completely described for nondiscrete, strongly rigid pseudometrics and, also, for discrete pseudometrics \(d \colon X^{2} \to \mathbb{R}\).
\end{abstract}

\maketitle

\section{Introduction}

Recall some definitions from the theory of metric spaces.

A \textit{metric} on a set \(X\) is a function \(d\colon X^{2} \to \RR\) such that for all \(x\), \(y\), \(z \in X\):
\begin{enumerate}
\item \(d(x,y) \geqslant 0\) with equality if and only if \(x=y\), the \emph{positivity property};
\item \(d(x,y)=d(y,x)\), the \emph{symmetry property};
\item \(d(x, y)\leq d(x, z) + d(z, y)\), the \emph{triangle inequality}.
\end{enumerate}

A useful generalization of the concept of metric is the concept of pseudometric.

\begin{definition}\label{ch2:d2}
Let \(X\) be a set and let \(d \colon X^{2} \to \RR\) be a nonnegative, symmetric function such that \(d(x, x) = 0\) holds for every \(x \in X\). The function \(d\) is a \emph{pseudometric} on \(X\) if it satisfies the triangle inequality.
\end{definition}

If \(d\) is a pseudometric on \(X\), we say that \((X, d)\) is a \emph{pseudometric} \emph{space}.

Let \((X, d)\) and \((Y, \rho)\) be pseudometric spaces. Then \((X, d)\) and \((Y, \rho)\) are \emph{isometric} if there is a bijective mapping \(\Phi \colon X \to Y\), an \emph{isometry} of \(X\) and \(Y\), such that
\[
\rho(\Phi(x), \Phi(y)) = d(x, y)
\]
holds for all \(x\), \(y \in X\). This classic (when \(d\) and \(\rho\) are metrics) concept has numerous generalizations. Recall two such generalizations which are closest to the object of the present research.

For pseudometric spaces \((X, d)\) and \((Y, \rho)\), a mapping \(\Phi \colon X \to Y\) is a \emph{similarity} if \(\Phi\) is bijective and there is \(r > 0\), the ratio of \(\Phi\), such that
\[
\rho(\Phi(x), \Phi(y)) = rd(x, y)
\]
holds for all \(x\), \(y \in X\). A bijective mapping \(F \colon X \to Y\) is a \emph{weak similarity} if there is a strictly increasing function \(f \colon \rho(Y^{2}) \to d(X^{2})\) such that the equality
\begin{equation}\label{e1.2}
d(x, y) = f(\rho(F(x), F(y)))
\end{equation}
holds for all \(x\), \(y \in X\). The function \(f\) is said to be a scaling function of \(F\).

It is clear that a similarity \(\Phi\) is an isometry if the ratio of \(\Phi\) is \(1\). Analogically, a weak similarity \(F\) is a similarity if the scaling function of \(F\) is linear.

The following definition can be considered as a further generalization of the concept of similarity.

\begin{definition}\label{d1.2}
Let \((X, d)\) and \((Y, \rho)\) be pseudometric spaces. The pseudometrics \(d\) and \(\rho\) are \emph{combinatorially similar} if there are bijections \(g \colon Y \to X\) and \(f \colon d(X^{2}) \to \rho(Y^{2})\) such that
\begin{equation}\label{d1.2:e1}
\rho(x,y) = f(d(g(x), g(y)))
\end{equation}
holds for all \(x\), \(y \in Y\). In this case, we will say that \(g \colon Y \to X\) is a \emph{combinatorial similarity}.
\end{definition}

It is easy to prove that, for every weak similarity, the scaling function is bijective (see \cite[Corollary~1.4]{DP} for the proof of this fact in the case of metric spaces). Consequently, we have the following chain of implications (see Figure~\ref{f1.1}).

Let us expand now the concept of combinatorially similar pseudometrics to the concept of \emph{combinatorially similar functions of two variables}.

\begin{definition}\label{d2.17}
Let \(X\), \(Y\) be sets and let \(\Phi\), \(\Psi\) be mappings with the domains \(X^{2}\), \(Y^{2}\) respectively. The mapping \(\Phi\) is \emph{combinatorially similar} to \(\Psi\) if there are bijections \(f \colon \Phi(X^2) \to \Psi(Y^{2})\) and \(g \colon Y \to X\) such that
\begin{equation}\label{d2.17:e1}
\Psi(x, y) = f(\Phi(g(x), g(y)))
\end{equation}
holds for all \(x\), \(y \in Y\).
\end{definition}

\begin{figure}[h]
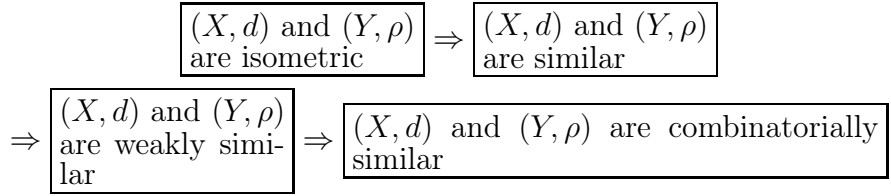

\begin{gather*}
\fbox{\parbox{3cm}{\((X, d)\) and \((Y, \rho)\) are isometric}} \Rightarrow \fbox{\parbox{3cm}{\((X, d)\) and \((Y, \rho)\) are similar}} \\
\Rightarrow \fbox{\parbox{3cm}{\((X, d)\) and \((Y, \rho)\) are weakly similar}} \Rightarrow \fbox{\parbox{7cm}{\((X, d)\) and \((Y, \rho)\) are combinatorially similar}}
\end{gather*}
\caption{From isometry to combinatorial similarity}\label{f1.1}
\end{figure}

Equality~\eqref{d2.17:e1} means that the diagram
\begin{equation}\label{eq1.3}
\ctdiagram{
\ctv 0,25:{X^{2}}
\ctv 100,25:{Y^{2}}
\ctv 0,-25:{\Phi(X^{2})}
\ctv 100,-25:{\Psi(Y^{2})}
\ctet 100,25,0,25:{g\otimes g}
\ctet 0,-25,100,-25:{f}
\ctel 0,25,0,-25:{\Phi}
\cter 100,25,100,-25:{\Psi}
}
\end{equation}
is commutative, where we use the denotation 
\[
(g\otimes g)(\<y_1, y_2>) := \<g(y_1), g(y_2)>
\]
with \(\<y_1, y_2> \in Y^{2}\) and write
\begin{align*}
\Psi(Y^{2}) &:= \{\Psi(y_1, y_2) \colon \<y_1, y_2> \in Y^{2}\},\\
\Phi(X^{2}) &:= \{\Phi(x_1, x_2) \colon \<x_1, x_2> \in X^{2}\}.
\end{align*}

The paper is organized as follows.

In Section~2 we briefly describe the well-known transition from pseudometrics to the metrics corresponding to them and introduce the basic concept of \(a_0\)-coherent mappings. This concept can be considered as an axiomatic description of those properties of pseudometrics which make such a transition correct. The main result of Section~\(2\) is Theorem~\ref{p2.7} describing the concept of \(a_0\)-coherent mappings in terms of semigroups of binary relations generated by fibers of these mappings. 

The main result of Section~3, Theorem~\ref{ch2:p7}, contains a necessary and sufficient condition under which a mapping is combinatorially similar to a pseudometric. 
Analyzing the proof of this theorem, we show that a mapping is combinatorially similar to a pseudometric if and only if it is combinatorially similar to a Ptolemaic pseudometric (Corollary~\ref{c2.20} and Corollary~\ref{c3.3}).
A combinatorial characterization of a discrete pseudometric is given in Theorem~\ref{t3.6}.
It is shown in Proposition~\ref{p3.8} that, for a finite set \(X\), the maximum number of discrete pseudometrics on \(X\) which are pairwise not combinatorially similar equals the number of distinct representations of \(|X|\) as a sum of positive integers.

The fourth section of the paper mainly deals with the strongly rigid pseudometrics. In Proposition~\ref{p4.4} and Theorem~\ref{p4.6} we characterize the mappings which are combinatorially similar to these pseudometrics. Proposition~\ref{p4.14} claims that, for combinatorially similar mappings \(\Phi\) and \(\Psi\) with \(\dom \Phi = X^{2}\) and \(\dom \Psi = Y^{2}\), the semigroups of binary relations \(\mathcal{B}_X (P_{\Phi^{-1}})\) and \(\mathcal{B}_Y (P_{\Psi^{-1}})\) generated by fibers of \(\Phi\) and \(\Psi\) are isomorphic.

For the case when \(d \colon X^{2} \to \RR\) is a strongly rigid and nondiscrete pseudometric, a purely algebraic characterization of \(\mathcal{B}_{X}(P_{d^{-1}})\) is presented in Theorem~\ref{t4.15}. For a discrete \(d\) the algebraic structure of \(\mathcal{B}_{X}(P_{d^{-1}})\) is completely described in Proposition~\ref{p4.22}. Theorem~\ref{ch2:th6} contains an algebraic characterization of the semigroup generated by all one-point subsets of \(X^{2}\).

\section{Pseudometrics and equivalence relations}

The main purpose of the present section is to describe a property of pseudometrics that characterizes them within a combinatorial similarity in Theorem~\ref{ch2:p7}.

Let \(X\) be a set. A \emph{binary relation} on \(X\) is a subset of the Cartesian square
\[
X^2 = X \times X = \{\<x, y>\colon x, y \in X\}.
\]
A binary relation \(R \subseteq X^{2}\) is an \emph{equivalence relation} on \(X\) if the following conditions hold for all \(x\), \(y\), \(z \in X\):
\begin{enumerate}
\item \(\<x, x> \in R\), the \emph{reflexivity} property;
\item \((\<x, y> \in R) \Leftrightarrow (\<y, x> \in R)\), the \emph{symmetry} property;
\item \(((\<x, y> \in R) \text{ and } (\<y, z> \in R)) \Rightarrow (\<x, z> \in R)\), the \emph{transitivity} property.
\end{enumerate}

If \(R\) is an equivalence relation on \(X\), then an \emph{equivalence class} is a subset \([a]_R\) of \(X\) having the form
\begin{equation}\label{e1.1}
[a]_R = \{x \in X \colon \<x, a> \in R\}
\end{equation}
for some \(a \in X\). The \emph{quotient set} of \(X\) with respect to \(R\) is the set of all equivalence classes \([a]_R\), \(a \in X\).

For every pseudometric space \((X, d)\), we define a binary relation \(\coherent{0}\) on \(X\) as
\begin{equation}\label{ch2:p1:e1}
(x \coherent{0} y) \Leftrightarrow (d(x, y) = 0),
\end{equation}
where, as usual, the formula \(x \coherent{0} y\) means that the ordered pair \(\<x, y>\) belongs to the relation \(\coherent{0}\).

The following proposition is an initial point of our consideration.

\begin{proposition}\label{ch2:p1}
Let \(X\) be a nonempty set and let \(d \colon X^{2} \to \RR\) be a pseudometric on \(X\). Then \({\coherent{0}}\) is an equivalence relation on \(X\) and the function \(\tilde{d}\),
\begin{equation}\label{ch2:p1:e2}
\tilde{d}(\alpha, \beta) := d(x, y), \quad x \in \alpha \in X/{\coherent{0}}, \quad y \in \beta \in X/{\coherent{0}},
\end{equation}
is a correctly defined metric on \(X/{\coherent{0}}\), where \(X/{\coherent{0}}\) is the quotient set of \(X\) with respect to \(\coherent{0}\).
\end{proposition}

The proof of Proposition~\ref{ch2:p1} can be found in \cite[Ch.~4, Th.~15]{Kelley1965}.
\medskip

In what follows we will sometimes say that \(\tilde{d} \colon (X / \coherent{0})^{2} \to \RR\) is the \emph{metric identification} of a pseudometric \(d \colon X^{2} \to \RR\).

Let us now consider some examples of pseudometrics and equivalence relations corresponding to them.

\begin{example}\label{ch2:ex6}
Let \(X\) be a nonempty set. The Cartesian square \(X^2\) is an equivalence relation on \(X\). It is easy to see that the \emph{zero pseudometric} \(d \colon X^{2} \to \RR\), \(d(x, y) = 0\) for all \(x\), \(y \in X\), is the unique pseudometric on \(X\) satisfying the equality
\[
X^{2} = \coherent{0}.
\]
\end{example}

\begin{example}\label{ch2:ex7}
Let \(X\) be a nonempty set and let 
\[
\Delta_X := \{\<x, x> \colon x \in X\}
\]
be the diagonal on \(X\). Then, for a pseudometric \(d \colon X^2 \to \RR\), the equality
\[
\Delta_X = \coherent{0}
\]
holds if and only if \(d\) is a metric.
\end{example}

\begin{example}\label{ex2.4}
Let \(X\) be a nonempty set and let \(\mathcal{PM}_X\) be the set of all pseudometrics on \(X\). For \(d_1\), \(d_2 \in \mathcal{PM}_X\) we write \(d_1 \approx d_2\) if \(d_1\) and \(d_2\) are combinatorially similar (see Definition~\ref{d1.2}). Then \(\approx\) is an equivalence relation on \(\mathcal{PM}_X\). Indeed, it follows directly from the definition that \(\approx\) is reflexive and symmetric. Now if we have \(d_1 \approx d_2\) and \(d_2 \approx d_3\), then the diagrams
\[
\ctdiagram{
\ctv 0,25:{X^{2}}
\ctv 100,25:{X^{2}}
\ctv 0,-25:{d_3(X^{2})}
\ctv 100,-25:{d_2(X^{2})}
\ctet 100,25,0,25:{g_2\otimes g_2}
\ctet 0,-25,100,-25:{f_2}
\ctel 0,25,0,-25:{d_3}
\cter 100,25,100,-25:{d_2}
} \qquad
\ctdiagram{
\ctv 0,25:{X^{2}}
\ctv 100,25:{X^{2}}
\ctv 0,-25:{d_2(X^{2})}
\ctv 100,-25:{d_1(X^{2})}
\ctet 100,25,0,25:{g_1\otimes g_1}
\ctet 0,-25,100,-25:{f_1}
\ctel 0,25,0,-25:{d_2}
\cter 100,25,100,-25:{d_1}
}
\]
are commutative, where \(g_1\), \(g_2\) and \(f_1\), \(f_2\) are defined as in~\eqref{eq1.3}. Consequently, the diagram
\[
\ctdiagram{
\ctv 0,25:{X^{2}}
\ctv 100,25:{X^{2}}
\ctv 200,25:{X^{2}}
\ctv 0,-25:{d_3(X^{2})}
\ctv 100,-25:{d_2(X^{2})}
\ctv 200,-25:{d_1(X^{2})}
\ctet 100,25,0,25:{g_2\otimes g_2}
\ctet 200,25,100,25:{g_1\otimes g_1}
\ctet 0,-25,100,-25:{f_2}
\ctet 100,-25,200,-25:{f_1}
\ctel 0,25,0,-25:{d_3}
\cter 200,25,200,-25:{d_1}
}
\]
is also commutative. It implies that \(\approx\) is transitive. Thus \(\approx\) is an equivalence relation on \(\mathcal{PM}_X\).

The equality \(\approx = \coherent{0}\) holds for the pseudometric \(\rho \colon (\mathcal{PM}_X)^2 \to \RR\) defined as
\[
\rho(d_1, d_2) := \begin{cases}
0, & \text{if } d_1 \approx d_2\\
1, & \text{otherwise}.
\end{cases}
\]
(An analogous formula could be used in the proof of Corollary~\ref{ch2:p2}.)
\end{example}

Let \(X\) be a nonempty set, let \(\Phi\) be a mapping with the domain \(X^2\) and let \(a_0 \in \Phi(X^2)\). In what follows we denote by \({\coherent{a_0}}\) the binary relation on \(X\) for which
\[
(x_1 \coherent{a_0} x_2) \Leftrightarrow (\Phi(x_1, x_2) = a_0)
\]
is valid for all \(x_1\), \(x_2 \in X\), i.e., 
\begin{equation}\label{ch2:d4:e1}
\coherent{a_0} = \Phi^{-1}(a_0).
\end{equation}

\begin{definition}\label{ch2:d4}
\(\Phi\) is \(a_0\)--\emph{coherent} if \(\coherent{a_0}\) is an equivalence relation and the implication
\begin{equation}\label{ch2:d4:e2}
\left((x_1 \coherent{a_0} x_2) \mathbin{\&} (x_3 \coherent{a_0} x_4)\right) \Rightarrow (\Phi(x_1, x_3) = \Phi(x_2, x_4))
\end{equation}
is valid for all \(x_1\), \(x_2\), \(x_3\), \(x_4 \in X\).
\end{definition}

The notion of \(a_0\)-coherent mappings can be considered as a special case of the notion of functions which are consistent with a given equivalence relation (see~\cite[p.~78]{KurMost}). 

\begin{example}\label{ch2:ex5}
From Proposition~\ref{ch2:p1} it follows that every pseudometric \(d\) is \(0\)-coherent. 
\end{example}

\begin{remark}\label{r2.9}
In the context of Proposition~\ref{ch2:p1} the statement
\begin{itemize}
\item \(d\) is \(0\)-coherent
\end{itemize}
means that
\begin{itemize}
\item the metric identification \(\tilde{d}\) of \(d\) is correctly defined by formula~\eqref{ch2:p1:e2}.
\end{itemize}
\end{remark}

\begin{example}\label{ch2:ex10}
Let \(X\) and \(Y\) be sets, let \(a_0 \in Y\) and let \(\Phi \colon X^{2} \to Y\) be a mapping such that
\[
\Phi^{-1}(a_0) = \Delta_X = \{\<x, x> \colon x \in X\}.
\]
Then \(\Phi\) is \(a_0\)-coherent.
\end{example}

\begin{example}\label{ch2:ex8}
Let \(X = \{0, 1\}\) and \(Y = \{a_0, a_1\}\) with \(a_0 \neq a_1\) and let \(\Phi \colon X^{2} \to Y\) be a function defined so that
\[
\Phi(x, y) = \begin{cases}
a_0, & \text{if } \<x, y> = \<0, 0>\\
a_1, & \text{if } \<x, y> \neq \<0, 0>.
\end{cases}
\]
It is easy to see that \eqref{ch2:d4:e2} is valid for all \(x_1\), \(x_2\), \(x_3\), \(x_4 \in X\) but \(\coherent{a_0}\) is not an equivalence relation on \(X\) because \(\coherent{a_0}\) is not reflexive. Consequently, \(\Phi\) is not \(a_0\)-coherent.
\end{example}

The following lemma shows that the property ``to be \(a_0\)-coherent for a suitable \(a_0\)'' is invariant under combinatorial similarities.

\begin{lemma}\label{l2.7}
Let \(X\), \(Y\) be nonempty sets and let \(\Phi\), \(\Psi\) be mappings with the domains \(X^{2}\), \(Y^{2}\) respectively. If the mapping \(\Phi\) is \(a_0\)-coherent and 
\[
f \colon \Phi(X^{2}) \to \Psi(Y^{2}), \quad g \colon Y \to X
\]
are bijections for which diagram~\eqref{eq1.3} is commutative, then the mapping \(\Psi\) is \(f(a_0)\)-coherent.
\end{lemma}

The proof is straightforward and we omit it here.
\medskip

The concept of combinatorial similarity can be also given in the language of the theory of semigroups.

Recall that a \emph{semigroup} is a pair \((\mathcal{S}, *)\) consisting of a nonempty set \(\mathcal{S}\) and an associative operation \(* \colon \mathcal{S} \times \mathcal{S} \to \mathcal{S}\), which is called the \emph{multiplication} on \(\mathcal{S}\).

A semigroup \(\mathcal{S} = (\mathcal{S}, *)\) is a \emph{monoid} if there is \(e \in \mathcal{S}\), then unique, such that
\[
e*s = s*e = s
\]
for every \(s \in \mathcal{S}\). In this case we say that \(e\) is the \emph{identity element} of the semigroup \((\mathcal{S}, *)\). A set \(G \subseteq \mathcal{S}\) is a \emph{set of generators} of \((\mathcal{S}, *)\) if, for every \(s \in \mathcal{S}\), there is a finite sequence \(g_1, \ldots, g_k\) of elements of \(G\) such that
\[
s = g_1 * \ldots * g_k.
\]

If \(A\) is a nonempty subset of a semigroup \((\mathcal{S}, *)\), \(A \subseteq \mathcal{S}\), we will denote by \(\mathcal{S}(A)\) the subsemigroup of \((\mathcal{S}, *)\) having \(A\) as a set of generators.

Let \(X\) be a set. The \emph{composition} \(\circ\) of binary relations \(\psi\), \(\gamma \subseteq X^{2}\) is a binary relation \(\psi \circ \gamma \subseteq X \times X\) for which \(\<x, y> \in \psi \circ \gamma\) holds if and only if there is \(z \in X\) such that \(\<x, z> \in \psi\) and \(\<z, y> \in \gamma\). It is well-known that the operation \(\circ\) is associative. Thus, the set \(\mathcal{B}_X\) of all binary relations on \(X\) together with the multiplication \(\circ\) is a semigroup.

Let \(Y\) be a nonempty set and let \(F\) be a mapping with the domain \(Y\). In what follows we will write
\begin{equation}\label{e4.4}
P_{F^{-1}} := \{F^{-1}(a) \colon a \in F(Y)\}
\end{equation}
for the set of fibers of \(F\). In particular, if \(Y\) is the Cartesian square of a set \(X\), \(Y = X^{2}\), then the set \(P_{F^{-1}}\) is a subset of \(\mathcal{B}_X\).

\begin{theorem}\label{p2.7}
Let \(X\) be a nonempty set, let \(\Phi\) be a mapping with the domain \(X^{2}\) and let \(a_0 \in \Phi(X^{2})\). Then the following statements are equivalent:
\begin{enumerate}
\item \label{p2.7:s1} \(\Phi\) is \(a_0\)-coherent. 
\item \label{p2.7:s2} \(\mathcal{B}_X(P_{\Phi^{-1}})\) is a monoid and \(\Phi^{-1}(a_0)\) is the identity element of this monoid.
\end{enumerate}
\end{theorem}

\begin{proof}
It is easy to see that \eqref{ch2:d4:e2} is valid for all \(x_1\), \(x_2\), \(x_3\), \(x_4 \in X\) if and, whenever \(\Phi^{-1}(a_0)\) is reflexive, only if
\begin{equation}\label{p2.7:e2}
\Phi^{-1}(a_0) \circ \Phi^{-1}(a) \subseteq \Phi^{-1}(a) \text{ and } \Phi^{-1}(a) \circ \Phi^{-1}(a_0) \subseteq \Phi^{-1}(a)
\end{equation}
is valid for every \(a \in \Phi(X^{2})\).

\(\ref{p2.7:s1} \Rightarrow \ref{p2.7:s2}\). Let \(\Phi\) be \(a_0\)-coherent. Then \(\Phi^{-1}(a_0)\) is an equivalence relation and \eqref{ch2:d4:e2} is valid for all \(x_1\), \(x_2\), \(x_3\), \(x_4 \in X\). Hence, \eqref{p2.7:e2} is valid for every \(a \in \Phi(X^{2})\). As \(\Phi^{-1}(a_0)\) is reflexive, we also have
\begin{equation}\label{p2.7:e3}
\Phi^{-1}(a_0) \circ \Phi^{-1}(a) \supseteq \Phi^{-1}(a) \text{ and } \Phi^{-1}(a) \circ \Phi^{-1}(a_0) \supseteq \Phi^{-1}(a)
\end{equation}
for every \(a \in \Phi(X^{2})\). From \eqref{p2.7:e2} and \eqref{p2.7:e3} it follows that
\begin{equation}\label{p2.7:e4}
\Phi^{-1}(a_0) \circ \Phi^{-1}(a) = \Phi^{-1}(a) \circ \Phi^{-1}(a_0) = \Phi^{-1}(a)
\end{equation}
for every \(a \in \Phi(X^{2})\). Thus, \(\mathcal{B}_X(P_{\Phi^{-1}})\) is a monoid with the identity element \(\Phi^{-1}(a_0)\). 

\(\ref{p2.7:s2} \Rightarrow \ref{p2.7:s1}\). Let \(\Phi^{-1}(a_0)\) be the identity element of \(\mathcal{B}_X(P_{\Phi^{-1}})\). We claim that the binary relation \(\Phi^{-1}(a_0)\) is reflexive. Suppose there is \(x_1 \in X\) for which
\begin{equation}\label{p2.7:e8}
\<x_1, x_1> \notin \Phi^{-1}(a_0)
\end{equation}
holds. Then we can find \(a_1 \in \Phi(X^{2})\) such that
\begin{equation}\label{p2.7:e9}
\<x_1, x_1> \in \Phi^{-1}(a_1).
\end{equation}
The equality \(\Phi^{-1}(a_1) \circ \Phi^{-1}(a_0) = \Phi^{-1}(a_1)\) implies that there exists \(x_2 \in X\) such that
\begin{equation}\label{p2.7:e10}
\<x_1, x_2> \in \Phi^{-1}(a_1) \text{ and } \<x_2, x_1> \in \Phi^{-1}(a_0).
\end{equation}
Now from the definition of composition of binary relations, \eqref{p2.7:e9}, \eqref{p2.7:e10} and the equality \(\Phi^{-1}(a_0) \circ \Phi^{-1}(a_1) = \Phi^{-1}(a_1)\) we obtain
\[
\<x_2, x_1> \in \Phi^{-1}(a_1).
\]
Consequently, we have the contradiction
\[
\<x_2, x_1> \in \Phi^{-1}(a_1) \cap \Phi^{-1}(a_0) = \varnothing.
\]
Hence, \(\Phi^{-1}(a_0)\) is reflexive.

Equalities \eqref{p2.7:e4} hold for every \(a \in \Phi(X^{2})\) as \(\Phi^{-1}(a_0)\) is the identity element of \(\mathcal{B}_X(P_{\Phi^{-1}})\). This implies \eqref{p2.7:e2} for every \(a \in \Phi(X^{2})\) and, consequently, \eqref{ch2:d4:e2} is valid for all \(x_1\), \(x_2\), \(x_3\), \(x_4 \in X\). For \(a = a_0\), from \eqref{p2.7:e4} it follows
\[
\Phi^{-1}(a_0) \circ \Phi^{-1}(a_0) \subseteq \Phi^{-1}(a_0),
\]
i.e., \(\Phi^{-1}(a_0)\) is transitive. To complete the proof it suffices to show that \(\Phi^{-1}(a_0)\) is symmetric. Suppose contrary that 
\begin{equation}\label{p2.7:e5}
\<x_1, x_2> \in \Phi^{-1}(a_0) \text{ but } \<x_2, x_1> \notin \Phi^{-1}(a_0)
\end{equation}
for some \(x_1\), \(x_2 \in X\). Then there is \(a_2 \in \Phi(X^{2})\) such that \(a_2 \neq a_0\) and 
\begin{equation}\label{p2.7:e6}
\<x_2, x_1> \in \Phi^{-1}(a_2).
\end{equation}
The definition of the composition \(\circ\) of binary relations and \eqref{p2.7:e5}, \eqref{p2.7:e6} imply
\begin{equation*}
\<x_1, x_1> \in \Phi^{-1}(a_0) \circ \Phi^{-1}(a_2) \subseteq \Phi^{-1}(a_2).
\end{equation*}
It follows from the reflexivity of \(\Phi^{-1}(a_0)\) that \(\<x_1, x_1> \in \Phi^{-1}(a_0)\). Consequently, we have the contradiction
\[
\<x_1, x_1> \in \Phi^{-1}(a_0) \cap \Phi^{-1}(a_2) = \varnothing. \qedhere
\]
\end{proof}

The uniqueness of the identity elements of semigroups and Theorem~\ref{p2.7} imply the following.

\begin{corollary}\label{c2.11}
Let \(X\) be a nonempty set and let \(\Phi\) be a mapping with the domain \(X^{2}\). If \(\Phi\) is \(a_0\)-coherent and \(a_1\)-coherent for \(a_0\), \(a_1 \in \Phi(X^{2})\), then \(a_0 = a_1\).
\end{corollary}

\begin{remark}\label{r2.12}
Corollary~\ref{c2.11} can be derived directly from Definition~\ref{ch2:d4}.
\end{remark}

\begin{example}\label{ex2.8}
Let \((G, *)\) and \((H, \circ)\) be groups, let \(e\) be the identity element of \((H, \circ)\) and let \(F \colon G \to H\) be a homomorphism of \((G, *)\) and \((H, \circ)\). Then the mapping
\[
\Phi \colon G^{2} \to H, \quad \<x, y> \mapsto F(x^{-1}*y)
\]
is \(e\)-coherent. Indeed, since \(F\) is a homomorphism, the set \(F^{-1}(e)\) is a normal subgroup of \(G\). Then
\[
\Phi^{-1}(e) = \{\<x, y> \in G^{2} \colon x^{-1}*y \in F^{-1}(e)\}
\]
is an equivalence relation. For every \(h \in H\), the set \(F^{-1}(h)\) is a coset of the subgroup \(F^{-1}(e)\) in the group \(G\). Now the validity of implication~\eqref{ch2:d4:e2} can be derived from the fact that the set of all cosets \(F^{-1}(h)\), \(h \in F(G)\), forms a factor group of \(G\) by the subgroup \(F^{-1}(e)\) with the standard multiplication
\[
(xF^{-1}(e)) (yF^{-1}(e)) = ((x*y)F^{-1}(e))
\]
(see, for example, \cite[Theorem~4.5]{Lang2005}).
\end{example}

\begin{example}\label{ex2.13}
Let \(X = \{0, 1\}\), and \(Y = \{a_0, a_1, a_2\}\) with \(|Y| = 3\), and \(\Phi \colon X^{2} \to Y\) be a function defined so that
\[
\Phi(x, y) = \begin{cases}
a_0, & \text{if } \<x, y> = \<0, 0>\\
a_1, & \text{if } \<x, y> = \<1, 1>\\
a_2, & \text{otherwise}.
\end{cases}
\]
Write \(P_{\Phi^{-1}} = \{\Phi^{-1}(a_0), \Phi^{-1}(a_1), \Phi^{-1}(a_2)\}\). It is easy to prove that
\[
\Phi^{-1}(a_2) \circ \Phi^{-1}(a_2) = \Delta_X \notin P_{\Phi^{-1}}.
\]
Since \(\Delta_X\) is the identity element of \(\mathcal{B}_X(P_{\Phi^{-1}})\), Theorem~\ref{p2.7} implies that \(\mathcal{B}_X(P_{\Phi^{-1}})\) is a monoid but there is no \(a \in \Phi(X^{2})\) such that \(\Phi\) is \(a\)-coherent.
\end{example}

\begin{example}\label{ch2:ex9}
Let \(X = \{0, 1, 2\}\) and \(Y = \{a_0, a_1, a_2\}\) be three-point sets and let \(\Phi \colon X^2 \to Y\) be a mapping such that 
\[
\Phi^{-1}(a_0)  = \Delta_X \cup \{\<0, 1>, \<1, 0>\}, \quad \Phi^{-1}(a_1) = \{\<2, 0>\}
\]
and
\[
\Phi^{-1}(a_2) = \{\<0, 2>, \<1,2>, \<2,1>\}.
\]
Then \(\coherent{a_0}\) is an equivalence relation on \(X\) but \(\Phi\) is not \(a_0\)-coherent, because 
\[
2 \coherent{a_0} 2 \quad \text{and} \quad 0 \coherent{a_0} 1 \quad \text{but} \quad \Phi(2, 0) \neq \Phi(2, 1).
\]
\end{example}

\begin{proposition}\label{ch2:p8}
Let \(\Phi \colon X^2 \to Y\) be a surjection and let \(a_0 \in Y\). If the inequality
\begin{equation}\label{ch2:p8:e1}
|Y| \leqslant 2
\end{equation}
holds, then the following statements are equivalent:
\begin{enumerate}
\item \label{ch2:p8:s1} \(\coherent{a_0}\) is an equivalence relation on \(X\).
\item \label{ch2:p8:s2} \(\Phi\) is \(a_0\)-coherent.
\end{enumerate}
\end{proposition}

The proof is straightforward and we omit it here.

\begin{remark}\label{ch2:r3}
Examples~\ref{ch2:ex10}, \ref{ch2:ex8}, \ref{ex2.13} and \ref{ch2:ex9} show that the conditions
\begin{itemize}
\item \({\coherent{a_0}}\) is an equivalence relation
\end{itemize}
and
\begin{itemize}
\item implication~\eqref{ch2:d4:e2} is valid for all \(x_1\), \(x_2\), \(x_3\), \(x_4 \in X\)
\end{itemize}
are logically independent and the constant \(2\) in inequality~\eqref{ch2:p8:e1} cannot be replaced by any integer which is strictly greater than \(2\).
\end{remark}

\section{Combinatorial similarity for general pseudometrics, Ptolemaic pseudometrics, and discrete pseudometrics}

The following result gives us a characterization of mappings which are combinatorially similar to pseudometrics.

\begin{theorem}\label{ch2:p7}
Let \(X\) be a nonempty set. The following conditions are equivalent for every mapping \(\Phi\) with the domain \(X^{2}\):
\begin{enumerate}
\item\label{ch2:p7:s1} \(\Phi\) is combinatorially similar to a pseudometric.
\item\label{ch2:p7:s2} \(\Phi\) is symmetric, there is \(a_0 \in \Phi(X^{2})\) such that \(\Phi\) is \(a_0\)-coherent, and the inequality \(|\Phi(X^{2})| \leqslant \mathfrak{c}\) holds, where \(\mathfrak{c}\) is the cardinality of the continuum.
\end{enumerate}
\end{theorem}

\begin{proof}
\(\ref{ch2:p7:s1} \Rightarrow \ref{ch2:p7:s2}\). Let \(\ref{ch2:p7:s1}\) hold. It suffices to show that \(\Phi\) is \(a_0\)-coherent for some \(a_0 \in \Phi(X^{2})\). This statement follows from Lemma~\ref{l2.7} and Example~\ref{ch2:ex5}.

\(\ref{ch2:p7:s2} \Rightarrow \ref{ch2:p7:s1}\). Let \(\Phi\) be symmetric and \(a_0\)-coherent, \(a_0 \in \Phi(X^{2})\) and let \(|\Phi(X^2)| \leqslant \mathfrak{c}\). Let us consider a subset \(D\) of the set \(\{0\} \cup [1, 2^{1/2}]\) such that \(0 \in D\) and \(|D| = |\Phi(X^{2})|\), where 
\[
[1, 2^{1/2}] = \{x \in \RR \colon 1 \leqslant x \leqslant 2^{1/2}\}.
\]
(We use \(2^{1/2}\) in place of \(2\), which would suffice here, for an application in the proof of Corollary~\ref{c2.20}.)

Let \(f \colon \Phi(X^{2}) \to D\) be a bijective mapping and let 
\begin{equation}\label{ch2:p7:e1}
f(a_0) = 0
\end{equation}
hold. Write \(d\) for the mapping such that
\begin{equation}\label{ch2:p7:e2}
d(x, y) = f(\Phi(x,y)), \quad x, y \in X.
\end{equation}
It is clear that \(d\) and \(\Phi\) are combinatorially similar. We claim that \(d\) is a pseudometric on \(X\). 

The function \(d\) is symmetric and nonnegative because \(\Phi\) is symmetric and \(f\) is nonnegative. Since \(\Phi\) is \(a_0\)-coherent, \(\Delta_X \subseteq \coherent{a_0}\) and \(f(a_0) = 0\), we have \(d(x, x) = 0\) for every \(x\in X\). Consequently \(d\) is a pseudometric if the triangle inequality
\begin{equation}\label{ch2:p7:e4}
d(x, y) \leqslant d(x, z) + d(z, y)
\end{equation}
holds for all \(x\), \(y\), \(z \in X\). If \(d(x, z)\) and \(d(z, y)\) are strictly positive, then, from the definitions of \(D\) and \(f\), we obtain
\[
1 \leqslant d(x, z) \text{ and } 1 \leqslant d(z, y) \text{ and } d(x, y) \leqslant 2^{1/2},
\]
which implies~\eqref{ch2:p7:e4}. Suppose \(d(x, z) = 0\) holds. Then, by definition of \(d\), we have \(x \coherent{a_0} z\). Since \(\Phi\) is \(a_0\)-coherent, we also have the equality \(\Phi(x, y) = \Phi(z, y)\). The last equality is equivalent to
\begin{equation}\label{ch2:p7:e5}
d(x, y) = d(z, y),
\end{equation}
because \(f\) is injective. Equality~\eqref{ch2:p7:e5} implies inequality~\eqref{ch2:p7:e4}. Analogously, from \(d(z, y) = 0\) we obtain \(d(x, y) = d(x, z)\). Inequality~\eqref{ch2:p7:e4} follows for all \(x\), \(y\), \(z \in X\). Thus \(d\) is a pseudometric on \(X\).

We also note that the equality \(f(a_0) = 0\) and \eqref{ch2:p7:e2} imply the equality
\[
{\coherent{a_0}} = {\coherent{0}}. \qedhere
\]
\end{proof}

Following Schoenberg~\cite{Schoenberg1940, Schoenberg1952} we call a pseudometric \(d \colon X^{2} \to \RR\) \emph{Ptolemaic} if the Ptolemy inequality 
\begin{equation}\label{e2.15}
d(x, z) d(y, t) + d(x, t) d(y, z) \geqslant d(x, y) d(z,t)
\end{equation}
holds for all \(x\), \(y\), \(z\), \(t \in X\).

\begin{corollary}\label{c2.20}
If a mapping \(\Phi\) is combinatorially similar to a pseudometric, then there is a Ptolemaic pseudometric \(d\) such that \(\Phi\) is combinatorially similar to \(d\).
\end{corollary}

\begin{proof}
Let \(\Phi\) be combinatorially similar to a pseudometric and let \(X\) be a nonempty set for which \(\dom \Phi = X^{2}\). Analyzing the proof of Theorem~\ref{ch2:p7} we can find a pseudometric \(d \colon X^{2} \to \RR\) such that \(\Phi\) and \(d\) are combinatorially similar and
\begin{equation}\label{c2.20:e1}
1 \leqslant d(x, y) \leqslant 2^{1/2},
\end{equation}
whenever \(d(x, y) \neq 0\). Consequently, if all distances between different points of the set \(\{x, y, z, t\}\) are nonzero, then we have
\[
d(x, y) d(z,t) \leqslant 2^{1/2} \cdot 2^{1/2} = 2
\]
and
\[
d(x, z) d(y, t) + d(x, t) d(y, z) \geqslant 1+1 = 2,
\]
which imply~\eqref{e2.15}. Moreover, since \(d\) is \(0\)-coherent and symmetric, from \(d(x,z) = 0\) it follows that
\[
d(x, t) d(y, z) = d(x, t)d(y, x) \quad \text{and} \quad d(x, y) d(z,t) = d(x, y) d(x,t).
\]
Hence, \eqref{e2.15} holds also for the case \(d(x,z) = 0\). Analogously, we can prove \eqref{e2.15} when \(d(y, t) = 0\) or \(d(x, t) = 0\) or \(d(y, z) = 0\). 

Thus, \(d\) is a Ptolemaic pseudometric.
\end{proof}

Let \(X\) be a set. A \emph{semimetric} on \(X\) is a nonnegative function \(d \colon X^{2} \to \RR\) such that \(d(x, y) = d(y, x)\) and \((d(x, y) = 0) \Leftrightarrow (x = y)\) for all \(x\), \(y \in X\) (see, for example, \cite[p.~7]{Blum(1953)}).

Example~\ref{ch2:ex7}, Example~\ref{ch2:ex10}, Theorem~\ref{ch2:p7} and Corollary~\ref{c2.20} imply the following.

\begin{corollary}\label{c3.3}
Let \(X\) be a nonempty set. The following conditions are equivalent for every mapping \(\Phi\) with the domain \(X^{2}\):
\begin{enumerate}
\item \label{c3.3:s1} \(\Phi\) is combinatorially similar to a Ptolemaic metric.
\item \label{c3.3:s2} \(\Phi\) is combinatorially similar to a metric.
\item \label{c3.3:s3} \(\Phi\) is combinatorially similar to a semimetric.
\item \label{c3.3:s4} \(\Phi\) is symmetric, \(|\Phi(X^2)| \leqslant \mathfrak{c}\) holds, and there is \(a_0 \in \Phi(X^{2})\) such that \(\Phi^{-1}(a_0) = \Delta_X\),  where \(\Delta_X\) is the diagonal on \(X\).
\end{enumerate}
\end{corollary}

Let \(\mathcal{A}_X\) be a subset of the set \(\mathcal{PM}_X (\mathcal{M}_X)\) of all pseudometrics (metrics) on a set \(X\). We say that \(\mathcal{A}_X\) is \emph{combinatorially universal} for \(\mathcal{PM}_X (\mathcal{M}_X)\) if for every \(d \in \mathcal{PM}_X(\mathcal{M}_X)\) there is \(\rho \in \mathcal{A}_X\) such that \(d\) is combinatorially similar to \(\rho\).

The following problem seems to be interesting.

\begin{problem}\label{pr2.22}
Find conditions under which a given subset of \(\mathcal{PM}_X(\mathcal{M}_X)\) is combinatorially universal for \(\mathcal{PM}_X (\mathcal{M}_X)\).
\end{problem}

We will not discuss this problem in details but note that, in accordance with Corollary~\ref{c2.20} (Corollary~\ref{c3.3}), the set of all Ptolemaic pseudometrics (metrics) \(d \colon X^{2} \to \RR\) is combinatorially universal for \(\mathcal{PM}_X\) \((\mathcal{M}_X)\).

A simple example of a set of pseudometrics which is not combinatorially universal is the class of all discrete pseudometrics.

We say that a pseudometric \(d \colon X^{2} \to \RR\) is \emph{discrete} if the inequality
\begin{equation}\label{e3.5}
|d(X^{2})| \leqslant 2
\end{equation}
holds. 

It is easy to see that a pseudometric \(d \colon X^{2} \to \RR\) is discrete if and only if there is \(k >0\) such that the implication
\begin{equation}\label{e3.5*}
(d(x, y) \neq 0) \Rightarrow (d(x, y) = k)
\end{equation}
is valid for all \(x\), \(y \in X\).

\begin{remark}\label{r3.2}
Implication~\eqref{e3.5*} is vacuously true for the zero pseudometrics. Hence, every zero pseudometric is discrete.
\end{remark}

\begin{proposition}\label{p3.3}
Let \(X\), \(Y\) be nonempty sets and let \(d\), \(\rho\) be discrete pseudometrics on \(X\) and \(Y\) respectively. Then, for every bijection \(g \colon Y \to X\), the following statements are equivalent:
\begin{enumerate}
\item \label{p3.3:s1} \(g\) is a similarity.
\item \label{p3.3:s2} \(g\) is a weak similarity.
\item \label{p3.3:s3} \(g\) is a combinatorial similarity.
\item \label{p3.3:s4} The logical equivalence 
\begin{equation}\label{p3.3:e1}
(\rho(x, y) = 0) \Leftrightarrow (d(g(x), g(y)) = 0)
\end{equation}
is valid for all \(x\), \(y \in Y\).
\end{enumerate}
\end{proposition}

\begin{proof}
The validity of the implications \(\ref{p3.3:s1} \Rightarrow \ref{p3.3:s2}\) and \(\ref{p3.3:s2} \Rightarrow \ref{p3.3:s3}\) has already been noted in the first section of the paper (see Figure~\ref{f1.1}).

\(\ref{p3.3:s3} \Rightarrow \ref{p3.3:s4}\). Let \(\ref{p3.3:s3}\) hold. By Example~\ref{ch2:ex5}, the pseudometric spaces \((X, d)\) and \((Y, \rho)\) are \(0\)-coherent. By Definition~\ref{d1.2}, there exists a bijection \(f \colon d(X^{2}) \to \rho(Y^{2})\) such that~\eqref{d1.2:e1} holds for all \(x\), \(y \in Y\). From Lemma~\ref{l2.7} it follows that the mapping \(\rho\) is \(f(0)\)-coherent. Now using Corollary~\ref{c2.11} we obtain the equality \(f(0) = 0\). The last equality and \eqref{d1.2:e1} imply the validity of \eqref{p3.3:e1} for all \(x\), \(y \in Y\).

\(\ref{p3.3:s4} \Rightarrow \ref{p3.3:s1}\). Let \(\ref{p3.3:s4}\) hold. Then from \eqref{p3.3:e1} it follows that
\begin{equation}\label{p3.3:e2}
\rho^{-1}(0) = \{\<x, y> \in Y^{2} \colon d(g(x), g(y)) = 0\}.
\end{equation}
If \(|\rho(Y^{2})| = 1\) then the last equality implies \(\ref{p3.3:s1}\) . Let \(|\rho(Y^{2})| = 2\). From \eqref{p3.3:e2} it follows that \(|d(X^{2})| \geqslant 2\). Using the definition of discrete pseudometrics we obtain the equality \(|d(X^{2})| = 2\). Consequently, we have also 
\begin{equation}\label{p3.3:e3}
\rho^{-1}(\rho(Y^{2})\setminus \{0\}) = \{\<x, y> \in Y^{2} \colon d(g(x), g(y)) \neq 0\}.
\end{equation}
Let \(r_d \in d(X^{2})\setminus \{0\}\) and \(r_{\rho} \in \rho(Y^{2})\setminus \{0\}\). Then from \eqref{p3.3:e1}, \eqref{p3.3:e2} and \eqref{p3.3:e3} it follows that \(g\) is a similarity of pseudometric spaces \((Y, \rho)\) and \((X, d)\) with the ratio \(r_d/r_{\rho}\).
\end{proof}

Theorem~\ref{ch2:p7} and Proposition~\ref{ch2:p8} imply that every equivalence relation coincides with \(\coherent{0}\) for a suitable discrete pseudometric.

\begin{corollary}\label{ch2:p2}
Let \(X\) be a nonempty set and let \(\equiv\) be an equivalence relation on \(X\). Then there is an up to a similarity unique discrete pseudometric \(d \colon X^{2} \to \RR\) such that \({\equiv} = \coherent{0}\).
\end{corollary}

\begin{proof}
Let \(\equiv\) be an equivalence relation on \(X\). Define a mapping \(\Phi \colon X^{2} \to \mathbb{R}\) by \(\Phi(x, y) = 0\) if \(x \equiv y\) and \(\Phi(x, y) = 1\) otherwise. Then \(\Phi\) is symmetric. As \(\Phi^{-1}(0)\) is an equivalence relation and \(|\Phi(X^{2})| \leqslant 2\), \(\Phi\) is \(0\)-coherent by Proposition~\ref{ch2:p8}. Hence, as \(|\Phi(X^{2})| \leqslant 2\), Theorem~\ref{ch2:p7} implies that \(\Phi\) is combinatorially similar to a pseudometric \(\rho\) on \(X\). Choose bijections \(f \colon \Phi(X^{2}) \to \rho(X^{2})\) and \(g \colon X \to X\) such that \(\rho(x, y) = f(\Phi(g(x), g(y)))\) for all \(x\), \(y \in X\). Then \(\rho\) is discrete. Define a discrete pseudometric \(d\) on \(X\) by \(d(x, y) = \rho(g^{-1}(x), g^{-1}(y)) = f(\Phi(x, y))\) for all \(x\), \(y \in X\). Fix \(x_0 \in X\). Then \(0 = d(x_0, x_0) = f(\Phi(x_0, x_0)) = f(0)\). Hence for all \(x\), \(y \in X\) we have \(x \equiv y \Leftrightarrow \Phi(x, y) = 0 \Leftrightarrow d(x, y) = 0\), as wanted.

To complete the proof it suffices to note that if \(d_1 \colon X^2 \to \RR\) and \(d_2 \colon X^2 \to \RR\) are discrete pseudometrics such that
\[
\{\<x,y> \in X^{2} \colon d_1(x, y) = 0\} = \{\<x,y> \in X^{2} \colon d_2(x, y) = 0\},
\]
then \(d_1\) and \(d_2\) are similar.
\end{proof}

\begin{remark}\label{r3.7}
The existence can also be shown by a simple independent proof. For this, the reader can extract the following hint from the given proof: \(d(x, y) = 0\) if \(x \equiv y\) and \(d(x, y) = f(1) > 0\) otherwise. Then the reader may assume that \(f(1) = 1\) and establish directly that \(d\) is a desired pseudometric (even a pseudoultrametric).
\end{remark}

Using Theorem~\ref{ch2:p7} and Corollary~\ref{ch2:p2} we obtain a combinatorial characterization of discrete pseudometrics.

\begin{theorem}\label{t3.6}
Let \(X\) be a nonempty set. The following conditions are equivalent for every mapping \(\Phi\) with the domain \(X^{2}\):
\begin{enumerate}
\item [\((i)\)] \(\Phi\) is combinatorially similar to a discrete pseudometric.
\item [\((ii)\)] There is \(a_0 \in \Phi(X^{2})\) such that \(\Phi\) is \(a_0\)-coherent and the inequality \(|\Phi(X^{2})| \leqslant 2\) holds.
\item [\((iii)\)] There is \(a_0 \in \Phi(X^{2})\) such that \(\Phi^{-1}(a_0)\) is an equivalence relation on \(X\) and the inequality \(|\Phi(X^{2})| \leqslant 2\) holds.
\end{enumerate}
\end{theorem}

\begin{proof}
Using the definition of discrete pseudometrics and Theorem~\ref{ch2:p7} we see that the logical equivalence \((i) \Leftrightarrow (ii)\) is valid. The implication \((ii) \Rightarrow (iii)\) is valid by Definition~\ref{ch2:d4}. 

Let \((iii)\) hold. Then \(\Phi^{-1}(a_0)\) is an equivalence relation on \(X\) for some \(a_0 \in \Phi(X^{2})\). From Corollary~\ref{ch2:p2} it follows that there is a discrete pseudometric \(d \colon X^{2} \to \RR\) such that \(\Phi^{-1}(a_0) = \coherent{0}\). This equality, the inequality \(|\Phi(X^{2})| \leqslant 2\) and the definition of discrete pseudometrics imply that \(\Phi\) and \(d\) are combinatorially similar. Thus, \((iii) \Rightarrow (i)\) also is valid.
\end{proof}

Suppose now that \(X\) is a finite, nonempty set and \(n = |X|\). If \(d \colon X^{2} \to \RR\) is a pseudometric, then we evidently have the equality
\begin{equation}\label{e3.12}
n = \sum_{\alpha \in X/ \coherent{0}} |\alpha|,
\end{equation}
where \(X/ \coherent{0}\) is the quotient set with respect to \(\coherent{0}\) (see Proposition~\ref{ch2:p1}). Thus, every pseudometric on \(X\) generates a representation of \(n\) as a sum of positive integers (\(=\) a \emph{partition of} \(n\)). Recall that a partition of a positive integer \(n\) is a finite sequence \(\{n_1, \ldots, n_k\}\) of positive integers such that \(n_1 \geqslant \ldots \geqslant n_k\) and
\[
n = \sum_{i=1}^{k} n_i.
\]

In what follows we denote by \(\mathcal{DP}_X\) the set of all discrete pseudometrics on a set \(X\) and write \(d_1 \approx d_2\) if \(d_1\), \(d_2 \in \mathcal{DP}_X\) and \(d_1\) is combinatorially similar to \(d_2\). By Example~\ref{ex2.4} the relation \(\approx\) is an equivalence relation on \(\mathcal{DP}_X\).

\begin{proposition}\label{p3.8}
Let \(X\) be a finite, nonempty set and \(n = |X|\). Let us denote by \(\mathcal{DP}_X / \approx\) the quotient set of \(\mathcal{DP}_X\) defined by \(\approx\). Then the cardinality \(|\mathcal{DP}_X / \approx|\) equals the number of distinct representations of \(n\) as a sum of positive integers.
\end{proposition}

\begin{proof}
Every \(d \in \mathcal{DP}_X\) generates a partition of \(n\) by formula~\eqref{e3.12}. It suffices to show that the partitions of \(n\) generated by arbitrary \(d_1\), \(d_2 \in \mathcal{DP}_X\) coincide if and only if \(d_1 \approx d_2\).

Let \(d_1\) and \(d_2\) belong to \(\mathcal{DP}_X\) and \(d_1 \approx d_2\). By Proposition~\ref{p3.3}, we have \(d_1 \approx d_2\) if and only if there is a bijection \(g \colon X \to X\) such that
\begin{equation}\label{p3.8:e2}
(d_1(x, y) = 0) \Leftrightarrow (d_2(g(x), g(y)) = 0)
\end{equation}
is valid for all \(x\), \(y \in X\). Write, for every \(x \in X\),
\begin{equation}\label{p3.8:e3}
[x]_1 := \{y \in X \colon d_{1}(x, y) = 0\},\ [x]_2 := \{y \in X \colon d_{2}(x, y) = 0\}.
\end{equation}
(Cf. \eqref{e1.1}.) Condition~\eqref{p3.8:e2} implies \(g([x]_1) = [g(x)]_2\) for every \(x \in X\). In particular, we have \(|[x]_1| = |[g(x)]_2|\) and \(|[x]_2| = |[g^{-1}(x)]_1|\) for every \(x \in X\). Consequently, \(d_1\) and \(d_2\) generate one and the same partition of \(n\). 

Conversely, let 
\[
\left\{n_1^{(s)}, \ldots, n_{k_s}^{(s)}\right\}, \quad s = 1, 2,
\]
be partitions of \(n\) such that 
\[
n_1^{(s)} = \left|[x_1^{(s)}]_s\right|,\ \ldots,\ n_{k_s}^{(s)} = \left|[x_{k_s}^{(s)}]_s\right|,
\]
where
\[
[x_j^{(s)}]_s \cap [x_i^{(s)}]_s = \varnothing
\]
for \(1 \leqslant i < j \leqslant k_s\) and \([x_j^{(1)}]_1\), \([x_j^{(2)}]_2\) are defined as in \eqref{p3.8:e3}. Suppose that 
\[
\left\{n_1^{(1)}, \ldots, n_{k_1}^{(1)}\right\} \text{ and } \left\{n_1^{(2)}, \ldots, n_{k_2}^{(2)}\right\}
\]
coincide as partitions of \(n\). Then \(k_1 = k_2\) and \(n_i^{(1)} = n_i^{(2)}\) holds for every \(i \in \{1, \ldots, k_1\}\). Consequently, there are bijections
\[
g_i \colon [x_i^{(1)}]_1 \to [x_i^{(2)}]_2, \quad i = 1, \ldots, k_1.
\]
Let us define \(g \colon X \to X\) such that the restriction of \(g\) on \([x_i^{(1)}]_1\) satisfies
\[
g|_{[x_i^{(1)}]_1} = g_i
\]
for every \(i \in \{1, \ldots, n_{k_1}\}\). Then \(g\) is a bijection satisfying~\eqref{p3.8:e2}.
\end{proof}

Using Proposition~\ref{p3.8} we can simply rewrite different results describing the behavior of partitions of positive integer numbers as some statements related to discrete pseudometrics on finite sets. For example, the famous Hardy--Ramanujan asymptotic formula
\[
p(n) \sim \frac{1}{4n\sqrt{3}} \exp \left(\pi \sqrt{\frac{2n}{3}}\right)
\]
for the number \(p(n)\) of the partitions of \(n\) (see \cite{And1976}, formula \((5.12)\)) gives.

\begin{corollary}\label{c3.9}
Let \(\{X_i\}_{i \in \mathbb{N}}\) be a sequence of finite sets such that
\[
\lim_{i \to \infty} |X_i| = \infty.
\]
Then the limit relation
\[
\lim_{i \to \infty} \frac{\left|X_i\right| \left|\mathcal{DP}_{X_i} / \approx\right|}{\exp \left(\pi \sqrt{\frac{2\left|X_i\right|}{3}}\right)} = \frac{1}{4\sqrt{3}}
\]
holds.
\end{corollary}

\section{Combinatorial similarity for strongly rigid pseudometrics}

Let \((X, d)\) be a metric space. The metric \(d\) is said to be \emph{strongly rigid} if, for all \(x\), \(y\), \(u\), \(v \in X\), the condition
\begin{equation*}
d(x, y) = d(u, v) \neq 0
\end{equation*}
implies
\begin{equation}\label{e3.6}
(x = u \text{ and } y = v) \text{ or } (x = v \text{ and } y = u).
\end{equation}
(See \cite{Janos1972, Martin1977} for characteristic topological properties of strongly rigid metric spaces.)

The concept of strongly rigid metric can be naturally generalized to the concept of \emph{strongly rigid pseudometric}.

\begin{definition}\label{d4.3}
Let \((X, d)\) be a pseudometric space. We say that the pseudometric \(d\) is strongly rigid if, for all \(x\), \(y\), \(u\), \(v \in X\), from
\begin{equation}\label{d4.3:e1}
d(x, y) = d(u, v) \neq 0
\end{equation}
it follows that
\begin{equation}\label{d4.3:e2}
(x \coherent{0} u \text{ and } y \coherent{0} v) \text{ or } (x \coherent{0} v \text{ and } y \coherent{0} u).
\end{equation}
\end{definition}

\begin{remark}\label{r4.2}
If \(\coherent{0} = \Delta_X\) holds, then \eqref{d4.3:e2} coincides with \eqref{e3.6}. Thus, a function \(d \colon X^{2} \to \RR\) is a strongly rigid metric if and only if \(d\) is a metric and a strongly rigid pseudometric.
\end{remark}

\begin{example}\label{ex4.3}
In the context of Proposition~\ref{ch2:p1} the statements ``the pseudometric \(d\) is strongly rigid'' and ``the metric identification \(\tilde{d}\) is strongly rigid'' are equivalent.
\end{example}

\begin{proposition}\label{p4.4}
Let \(X\) be a nonempty set. Then, for every mapping \(\Phi\) with the domain \(X^{2}\), the following statements are equivalent:
\begin{enumerate}
\item \label{p4.4:s1} There is a strongly rigid, Ptolemaic pseudometric \(d\) such that \(\Phi\) is combinatorially similar to \(d\).
\item \label{p4.4:s2} There is a strongly rigid pseudometric \(d\) such that \(\Phi\) is combinatorially similar to \(d\).
\item \label{p4.4:s3} \(\Phi\) is symmetric, the inequality \(|\Phi(X^{2})| \leqslant \mathfrak{c}\) holds, there is \(a_0 \in \Phi(X^{2})\) such that \(\Phi\) is \(a_0\)-coherent, and, for all \(x\), \(y\), \(u\), \(v \in X\), the condition
\[
\Phi(x, y) = \Phi(u, v) \neq a_0
\]
implies
\[
(x \coherent{a_0} u \text{ and } y \coherent{a_0} v) \text{ or } (x \coherent{a_0} v \text{ and } y \coherent{a_0} u).
\]
\end{enumerate}
\end{proposition}

\begin{proof}
\(\ref{p4.4:s1} \Rightarrow \ref{p4.4:s2}\). This is trivial.

\(\ref{p4.4:s2} \Rightarrow \ref{p4.4:s3}\). Let \ref{p4.4:s2} hold. Then \ref{p4.4:s3} follows from Definition~\ref{d2.17}, Theorem~\ref{ch2:p7} and Definition~\ref{d4.3}.

\(\ref{p4.4:s3} \Rightarrow \ref{p4.4:s1}\). Suppose \ref{p4.4:s3} holds. As in Corollary~\ref{c2.20}, we see that the pseudometric \(d \colon X^{2} \to \RR\), for which \eqref{ch2:p7:e1} and \eqref{ch2:p7:e2} hold, is Ptolemaic and combinatorially similar to \(\Phi\). Now, from condition \ref{p4.4:s3} and equality \eqref{ch2:p7:e2} it follows that \(d\) is strongly rigid.
\end{proof}

Using Proposition~\ref{p4.4} and Remark~\ref{r4.2} we obtain the following.

\begin{corollary}\label{c4.4}
Let \(X\) be a nonempty set. Then, for every mapping \(\Phi\) with the domain \(X^{2}\), the following conditions are equivalent:
\begin{enumerate}
\item \label{c4.4:s1} There is a strongly rigid, Ptolemaic metric \(d\) such that \(\Phi\) is combinatorially similar to \(d\).
\item \label{c4.4:s2} There is a strongly rigid metric \(d\) such that \(\Phi\) is combinatorially similar to \(d\).
\item \label{c4.4:s3} \(\Phi\) is symmetric, the inequality \(|\Phi(X^{2})| \leqslant \mathfrak{c}\) holds, there is \(a_0 \in \Phi(X^{2})\) such that \(\Phi^{-1}(a_0) = \Delta_X\), and, for all \(x\), \(y\), \(u\), \(v \in X\), the condition
\[
\Phi(x, y) = \Phi(u, v) \neq a_0
\]
implies
\[
(x = u \text{ and } y = v) \text{ or } (x = v \text{ and } y = u).
\]
\end{enumerate}
\end{corollary}

Let \(X\) be a nonempty set and \(P = \{X_j \colon j \in J\}\) be a set of nonempty subsets of \(X\). Recall that \(P\) is a \emph{partition} of \(X\) if we have
\[
\bigcup_{j \in J} X_j = X \quad \text{and} \quad X_{j_1} \cap X_{j_2} = \varnothing
\]
for all distinct \(j_1\), \(j_2 \in J\). In what follows we will say that the sets \(X_j\), \(j \in J\), are the \emph{blocks} of \(P\).

\begin{example}\label{ex4.6}
Let \(Y\) be a nonempty set and let \(F\) be a mapping with \(\dom F = Y\). Then \(P_{F^{-1}} = \{F^{-1}(a) \colon a \in F(Y)\}\) is a partition of \(Y\) whose blocks are the fibers of \(F\).
\end{example}

If \(P = \{X_j \colon j \in J\}\) and \(Q = \{X_i \colon i \in I\}\) are partitions of a set \(X\), then we say that \(P\) and \(Q\) are equal, and write \(P = Q\), if and only if there is a bijective mapping \(f \colon J \to I\) such that \(X_j = X_{f(j)}\) for every \(j \in J\).

There exists a well-known, one-to-one correspondence between the equivalence relations and partitions.

\begin{proposition}\label{p4.7}
Let \(X\) be a nonempty set. If \(P = \{X_j \colon j \in J\}\) is a partition of \(X\) and \(R_P\) is a binary relation on \(X\) defined as
\begin{itemize}
\item[] \(\<x, y> \in R_P\) if and only if \(\exists j \in J\) such that  \(x \in X_j\) and \(y \in X_j\),
\end{itemize}
then \(R_P\) is an equivalence relation on \(X\) with the equivalence classes \(X_j\). Conversely, if \(R\) is an equivalence relation on \(X\), then the set \(P_R\) of all distinct equivalence classes \([a]_R\) is a partition of \(X\) with the blocks \([a]_R\).
\end{proposition}

For the proof, see, for example, \cite[Chapter~II, \S{}~5]{KurMost}.

\begin{lemma}[{\cite[p.~9]{Kelley1965}}]\label{l4.5}
Let \(X\) be a nonempty set. If \(R\) is an equivalence relation on \(X\) and \(P_R = \{X_j \colon j \in J\}\) is the corresponding partition of \(X\), then the equality
\begin{equation}\label{l4.5:e1}
R = \bigcup_{j \in J} X_j^2
\end{equation}
holds.
\end{lemma}

A proof of this lemma can be found in \cite[Proposition~2.6]{Dovgoshey2019}.

For every partition \(P = \{X_j \colon j \in J\}\) of a nonempty set \(X\) we define a partition \(P \otimes P_S^1\) of \(X^{2}\) by the rule: 
\begin{itemize}
\item A subset \(B\) of \(X^{2}\) is a block of \(P \otimes P_S^1\) if and only if either 
\[
B = \bigcup_{j \in J} X_{j}^{2}
\]
or there are \emph{distinct} \(j_1\), \(j_2 \in J\) such that
\[
B = (X_{j_1} \times X_{j_2}) \cup (X_{j_2} \times X_{j_1}).
\]
\end{itemize}

\begin{remark}\label{r4.5}
If \(R_P\) is the equivalence relation corresponding to a partition \(P = \{X_j \colon j \in J\}\) of \(X\), then using Lemma~\ref{l4.5} we see that \(R_P\) belongs to the partition \(P \otimes P_S^{1}\).
\end{remark}

\begin{lemma}\label{l4.9}
Let \(X\) be a nonempty set and let \(\Phi\), \(\Psi\) be mappings such that \(\dom \Phi = \dom \Psi = X^{2}\). If the equality 
\begin{equation}\label{l4.9:e1}
P_{\Phi^{-1}} = P_{\Psi^{-1}}
\end{equation}
holds, then \(\Phi\) and \(\Psi\) are combinatorially similar.
\end{lemma}

\begin{proof}
Let
\[
P_{\Phi^{-1}} = \{\Phi^{-1}(a) \colon a \in \Phi(X^{2})\} \quad \text{and} \quad  P_{\Psi^{-1}} = \{\Psi^{-1}(a) \colon a \in \Psi(X^{2})\}
\]
be equal. Then there is a bijection \(f \colon \Phi(X^{2}) \to \Psi(X^{2})\) for which the diagram 
\[
\ctdiagram{
\ctv 50,30:{X^{2}}
\ctv 0,-30:{\Phi(X^{2})}
\ctv 100,-30:{\Psi(X^{2})}
\ctel 50,30,0,-30:{\Phi}
\cter 50,30,100,-30:{\Psi}
\ctet 0,-30,100,-30:{f}
}
\]
is commutative. It implies the commutativity of diagram~\eqref{eq1.3} with \(X = Y\) and \(g\) the identity mapping of \(X\).
\end{proof}

\begin{definition}\label{d4.10}
Let \(X\) be a nonempty set and let \(Q\) be a partition of the set \(X^{2}\). Then \(Q\) is \emph{symmetric} if the equivalence 
\[
\bigl(\<x, y> \in B\bigr) \Leftrightarrow \bigl(\<y, x> \in B\bigr)
\]
is valid for each block \(B\) of \(Q\) and every \(\<x, y> \in X^{2}\). 
\end{definition}

Thus, \(Q\) is a symmetric partition of \(X^{2}\) if every block of \(Q\) is a symmetric binary relation on \(X\).

\begin{lemma}[{\cite[Lemma~2.21]{Dovgoshey2019}}]\label{l4.8}
Let \(X\) be a nonempty set and let \(\Phi\) be a mapping with the domain \(X^{2}\). Then the mapping \(\Phi\) is symmetric if and only if the partition \(P_{\Phi^{-1}}\) is symmetric. 
\end{lemma}

\begin{theorem}\label{p4.6}
Let \(X\) be a nonempty set and let \(\Phi\) be a mapping with the domain \(X^{2}\). Then the following conditions are equivalent:
\begin{enumerate}
\item\label{p4.6:s1} There is a strongly rigid pseudometric \(d \colon X^{2} \to \RR\) such that \(P_{\Phi^{-1}} = P_{d^{-1}}\).
\item\label{p4.6:s2} \(\Phi\) is combinatorially similar to a strongly rigid pseudometric.
\item\label{p4.6:s3} There is a partition \(P\) of \(X\) such that the partitions \(P \otimes P_{S}^{1}\) and \(P_{\Phi^{-1}}\) are equal, \(P \otimes P_{S}^{1} = P_{\Phi^{-1}}\), and the inequality \(|P| \leqslant \mathfrak{c}\) holds.
\end{enumerate}
\end{theorem}

\begin{proof}
\(\ref{p4.6:s1} \Leftrightarrow \ref{p4.6:s2}\). The implication \(\ref{p4.6:s1} \Rightarrow \ref{p4.6:s2}\) is valid by Lemma~\ref{l4.9}. Let \(\Phi\) be combinatorially similar to a strongly rigid pseudometric \(\rho \colon Y^{2} \to \RR\). Then using the definition of combinatorial similarity with \(\Psi = \rho\) we obtain the commutative diagram
\begin{equation*}
\ctdiagram{
\ctv 0,25:{X^{2}}
\ctv 100,25:{Y^{2}}
\ctv 0,-25:{\Phi(X^{2})}
\ctv 100,-25:{\rho(Y^{2})}
\ctet 100,25,0,25:{g\otimes g}
\ctet 0,-25,100,-25:{f}
\ctel 0,25,0,-25:{\Phi}
\cter 100,25,100,-25:{\rho}
}
\end{equation*}
with bijections \(f\) and \(g\). Then the following diagram is also commutative:
\begin{equation}\label{p4.6:e0}
\ctdiagram{
\ctv 0,25:{X^{2}}
\ctv 100,25:{Y^{2}}
\ctv 0,-25:{\Phi(X^{2})}
\ctv 100,-25:{\rho(Y^{2})}
\ctet 0,25,100,25:{g^{-1}\otimes g^{-1}}
\ctet 100,-25,0,-25:{f^{-1}}
\ctel 0,25,0,-25:{\Phi}
\cter 100,25,100,-25:{\rho}
}
\end{equation}
(\(\Phi\) and \(\rho\) are combinatorially similar if and only if \(\rho\) and \(\Phi\) are combinatorially similar). Write \(d\) for
\[
X^{2} \xrightarrow{g^{-1}\otimes g^{-1}} Y^{2} \xrightarrow{\rho} \RR.
\]
Then \(d\) is a strongly rigid pseudometric on \(X\). The equality
\[
P_{\Phi^{-1}} = P_{d^{-1}}
\]
follows from the commutativity of~\eqref{p4.6:e0}. Hence, \(\ref{p4.6:s2} \Rightarrow \ref{p4.6:s1}\) is true. The validity of \(\ref{p4.6:s1} \Leftrightarrow \ref{p4.6:s2}\) follows.

\(\ref{p4.6:s1} \Rightarrow \ref{p4.6:s3}\). Let \(d \colon X^{2} \to \RR\) be a strongly rigid pseudometric on \(X\) satisfying the equalities
\[
P_{\Phi^{-1}} = P_{d^{-1}} = \{d^{-1}(r) \colon r \in d(X^{2})\}.
\]

Let us denote by \(P = \{X_j \colon j \in J\}\) the partition of \(X\) corresponding to the equivalence relation \(\coherent{0}\) generated by \(d\) (see formula \eqref{ch2:p1:e1}). We claim that \(P_{d^{-1}}\) and \(P \otimes P_{S}^{1}\) are equal as partitions of \(X^{2}\),
\begin{equation}\label{p4.6:e1}
P_{d^{-1}} = P \otimes P_{S}^{1}.
\end{equation}
Since \(d^{-1}(0) = \coherent{0}\) is a block of \(P_{d^{-1}}\), from Lemma~\ref{l4.5} it follows that
\[
\bigcup_{j \in J} X_j^{2} \in P_{d^{-1}}.
\]
Suppose now \(r \in d(X^{2})\setminus \{0\}\) and consider the block \({\coherent{r}} = d^{-1}(r)\) of \(P_{d^{-1}}\). Then there are distinct \(x\), \(y \in X\) such that
\begin{equation}\label{p4.6:e2}
d(x, y) = d(y, x) = r > 0.
\end{equation}
Since \(d\) is \(0\)-coherent, Theorem~\ref{p2.7}, Lemma~\ref{l4.5} and condition~\eqref{p4.6:e2} imply
\begin{align*}
\coherent{r} &= \coherent{0} \circ \coherent{r} \circ \coherent{0} = \left(\bigcup_{j \in J} X_j^{2}\right) \circ \coherent{r} \circ \left(\bigcup_{j \in J} X_j^{2}\right) \\
& \supseteq \left(\bigcup_{j \in J} X_j^{2}\right) \circ (\{\<x, y>, \<y, x>\}) \circ \left(\bigcup_{j \in J} X_j^{2}\right)\\
& = (X_{j_x} \times X_{j_y}) \cup (X_{j_y} \times X_{j_x}),
\end{align*}
where \(X_{j_x}\) and \(X_{j_y}\) are the blocks of \(P_{d^{-1}}\) such that \(x \in X_{j_x}\) and \(y \in X_{j_y}\). Consequently, we have
\[
\coherent{r} \supseteq (X_{j_x} \times X_{j_y}) \cup (X_{j_y} \times X_{j_x}).
\]
If \(\<u, v>\) is an arbitrary point of the block \(\coherent{r}\), then \(d(x, y) = d(u, v)\) holds. This equality and~\eqref{p4.6:e2} imply \eqref{d4.3:e2}. Now using Proposition~\ref{p4.7} and condition \eqref{d4.3:e2} we obtain
\[
(u \in X_{j_x} \text{ and } v \in X_{j_y}) \text{ or } (u \in X_{j_y} \text{ and } v \in X_{j_x}).
\]
Hence, 
\[
\<u, v> \in (X_{j_x} \times X_{j_y}) \cup (X_{j_y} \times X_{j_x})
\]
holds for all \(\<u, v> \in \coherent{r}\). It implies
\[
\coherent{r} = (X_{j_x} \times X_{j_y}) \cup (X_{j_y} \times X_{j_x}).
\]
Thus, the partition \(P_{d^{-1}}\) is a subset of the partition \(P \otimes P_S^{1}\). It should be noted that if we have two partitions of the same set and one of these partitions is a subset of the other, then the partitions are equal. Equality~\eqref{p4.6:e1} follows.

To complete the proof of validity of \(\ref{p4.6:s1} \Rightarrow \ref{p4.6:s3}\) we note that the inequality \(|P| \leqslant \mathfrak{c}\) follows from
\[
|P \otimes P_S^{1}| = |P_{d^{-1}}| = |\{d^{-1}(z) \colon z \in d(X^{2})\}|
\]
and the inequality \(|d(X^{2})| \leqslant \mathfrak{c}\).

\(\ref{p4.6:s3} \Rightarrow \ref{p4.6:s2}\). Let \(P := \{X_j \colon j \in J\}\) be a partition of \(X\) such that
\[
P \otimes P_S^{1} = P_{\Phi^{-1}} \quad \text{and} \quad |P| \leqslant \mathfrak{c}.
\]
Let us define a mapping \(\Psi \colon X^{2} \to P \otimes P_{S}^{1}\) such that
\[
(\Psi(\<x, y>) = b) \Leftrightarrow (\<x,y> \in b)
\]
for every ordered pair \(\<x,y> \in X^{2}\) and every block \(b \in P \otimes P_{S}^{1}\). (If \(R_{P \otimes P_{S}^{1}}\) is the equivalence relation on \(X^{2}\) corresponding to the partition \(P \otimes P_{S}^{1}\) of \(X^{2}\), then \(P \otimes P_{S}^{1}\) is the quotient set of \(X^{2}\) with respect to \(R_{P \otimes P_{S}^{1}}\) and \(\Psi\) is the natural projection of \(X^{2}\) onto this quotient set.) It is clear that \(P_{\Psi^{-1}} = P \otimes P_{S}^{1}\) holds. Consequently, we have \(P_{\Psi^{-1}} = P_{\Phi^{-1}}\). By Lemma~\ref{l4.9}, the mappings \(\Phi\) and \(\Psi\) are combinatorially similar. Thus, it suffices to show that \(\Psi\) is combinatorially similar to a strongly rigid pseudometric. 

By Lemma~\ref{l4.8}, the mapping \(\Psi\) is symmetric, because the partition \(P \otimes P_{S}^{1}\) is evidently symmetric. Since \(\Psi\) and \(\Phi\) are combinatorially similar, the inequality \(|\Phi(X^{2})| \leqslant \mathfrak{c}\) implies the inequality \(|\Psi(X^{2})| \leqslant \mathfrak{c}\). In accordance with Remark~\ref{r4.5}, we have \(R_{P} \in P \otimes P_{S}^{1}\). From the definition of \(P \otimes P_{S}^{1}\) and the definition of composition of binary relations it follows that \(R_{P}\) is the identity element of the semigroup \(\mathcal{B}_X(A)\) with
\[
A := \{\Psi^{-1}(b) \colon b \in \Psi(X^{2})\}.
\]
Hence, by Theorem~\ref{p2.7}, the mapping \(\Psi\) is \(b_0\)-coherent with \(b_0 = R_{P}\).

Now suppose that
\begin{equation}\label{p4.6:e3}
\Psi(x, y) = \Psi(u, v) \neq R_{P}.
\end{equation}
Then using the definitions of \(\Psi\) and \(P \otimes P_{S}^{1}\) we can find some distinct \(j_1\), \(j_2 \in J\) such that
\[
\<x, y> \in (X_{j_1} \times X_{j_2}) \cup (X_{j_2} \times X_{j_1})
\]
and 
\[
\<u, v> \in (X_{j_1} \times X_{j_2}) \cup (X_{j_2} \times X_{j_1}).
\]
Suppose \(\<x, y> \in X_{j_1} \times X_{j_2}\). Then we obtain either
\begin{equation}\label{p4.6:e4}
u \coherent{b_0} x \text{ and } v \coherent{b_0} y
\end{equation}
if \(\<u, v> \in X_{j_1} \times X_{j_2}\) or 
\begin{equation}\label{p4.6:e5}
u \coherent{b_0} y \text{ and } v \coherent{b_0} x
\end{equation}
if \(\<u, v> \in X_{j_2} \times X_{j_1}\). Analogously, for \(\<x, y> \in X_{j_2} \times X_{j_1}\), we also have either \eqref{p4.6:e4} or \eqref{p4.6:e5}. Hence \(\Psi\) is combinatorially similar to a strongly rigid pseudometric by Proposition~\ref{p4.4}.
\end{proof}

\begin{corollary}\label{c4.14}
Let \(X\) be a nonempty set and let \(d \colon X^{2} \to \RR\) be a pseudometric. Then \(d\) is strongly rigid if and only if 
\[
P_{d^{-1}} = P \otimes P_{S}^{1}
\]
holds with the partition \(P\) of \(X\) corresponding to the equivalence relation \(\coherent{0}\) generated by \(d\).
\end{corollary}

\begin{proof}
The ``only if'' claim follows from Theorem~\ref{p4.6} and formula~\eqref{p4.6:e1} in its proof. For the ``if'' claim, note that Theorem~\ref{p4.6} yields a strongly rigid pseudometric \(\rho\) on \(X\) combinatorially similar to \(d\), which implies that \(d\), too, is strongly rigid.
\end{proof}

Let us recall the concept of isomorphic semigroups.

\begin{definition}\label{d4:13}
Let \((\mathcal{S}, \circ)\) and \((\mathcal{H}, *)\) be semigroups. A mapping \(F \colon \mathcal{S} \to \mathcal{H}\) is a \emph{homomorphism} if
\[
F(s_1 \circ s_2) = F(s_1) * F(s_2)
\]
holds for all \(s_1\), \(s_2 \in \mathcal{S}\). A bijective homomorphism is an \emph{isomorphism}.

Semigroups \(\mathcal{S}\) and \(\mathcal{H}\) are \emph{isomorphic} if there is an isomorphism \(F \colon \mathcal{S} \to \mathcal{H}\).
\end{definition}

\begin{proposition}\label{p4.14}
Let \(X\), \(Y\) be nonempty sets and let \(\Phi\), \(\Psi\) be mappings with the domains \(X^{2}\) and \(Y^{2}\), respectively. If \(\Phi\) and \(\Psi\) are combinatorially similar, then the semigroups \(\mathcal{B}_X(P_{\Phi^{-1}})\) and \(\mathcal{B}_Y(P_{\Psi^{-1}})\) are isomorphic, where \(P_{\Phi^{-1}}\) and \(P_{\Psi^{-1}}\) are the partitions of \(X^{2}\) and, respectively, of \(Y^{2}\) defined as in~\eqref{e4.4}.
\end{proposition}

\begin{proof}
Let \(\Phi\) and \(\Psi\) be combinatorially similar. Then there are bijections \(f\colon \Phi(X^{2}) \to \Psi(Y^{2})\) and \(g \colon Y \to X\) such that
\begin{equation}\label{p4.14:e1}
\Psi(\<x,y>) = f(\Phi(\<g(x), g(y)>))
\end{equation}
holds for every \(\<x,y> \in Y^{2}\) (see diagram~\eqref{eq1.3}). It is easy to prove that the mapping \(\widehat{g} \colon \mathcal{B}_Y \to \mathcal{B}_X\) defined as
\[
\widehat{g}(A) := \{\<g(x), g(y)> \colon \<x, y> \in A\}
\]
for \(A \subseteq Y^{2}\) is an isomorphism of the semigroups \(\mathcal{B}_X\) and \(\mathcal{B}_Y\). Moreover, since \(f\) is bijective, the partitions \(P_{\Phi^{-1}}\) and \(P_{(f \circ \Phi)^{-1}}\) of \(X^{2}\) are equal, \(P_{\Phi^{-1}} = P_{(f \circ \Phi)^{-1}}\), where \(f \circ \Phi\) is the mapping
\[
X^{2} \xrightarrow{\Phi} \Phi(X^{2}) \xrightarrow{f} \Psi(Y^{2}).
\]
Consequently, the commutativity of diagram~\eqref{eq1.3} implies that the set \(P_{\Phi^{-1}}\) is the image of the set \(P_{\Psi^{-1}}\) under the mapping \(\widehat{g} \colon \mathcal{B}_Y \to \mathcal{B}_X\). Since \(\widehat{g}\) is an isomorphism of \(\mathcal{B}_Y\) and \(\mathcal{B}_X\) and, moreover, \(P_{\Psi^{-1}}\) and \(P_{\Phi^{-1}}\) are the sets of generators of \(\mathcal{B}_Y(P_{\Psi^{-1}})\) and \(\mathcal{B}_X(P_{\Phi^{-1}})\), respectively, the semigroups \(\mathcal{B}_Y(P_{\Psi^{-1}})\) and \(\mathcal{B}_X(P_{\Phi^{-1}})\) are isomorphic.
\end{proof}

\begin{figure}[h]
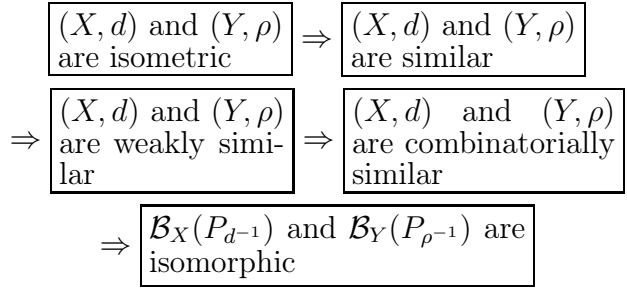

\begin{gather*}
\fbox{\parbox{3cm}{\((X, d)\) and \((Y, \rho)\) are isometric}} \Rightarrow \fbox{\parbox{3cm}{\((X, d)\) and \((Y, \rho)\) are similar}} \\
\Rightarrow \fbox{\parbox{3cm}{\((X, d)\) and \((Y, \rho)\) are weakly similar}} \Rightarrow \fbox{\parbox{3.5cm}{\((X, d)\) and \((Y, \rho)\) are combinatorially similar}} \\
\Rightarrow \fbox{\parbox{5cm}{\(\mathcal{B}_X(P_{d^{-1}})\) and \(\mathcal{B}_Y(P_{\rho^{-1}})\) are isomorphic}}
\end{gather*}
\caption{From isometric pseudometric spaces to isomorphic semigroups}\label{f2}
\end{figure}

In accordance with Proposition~\ref{p4.14}, the chain of implications from Figure~\ref{f1.1} can be extended to the chain in Figure~\ref{f2}.

In the following proposition we use the notation of Proposition~\ref{ch2:p1}.

\begin{proposition}\label{p4.16}
Let \(X\) be a nonempty set, \(d \colon X^{2} \to \RR\) a pseudometric on \(X\) and \(\tilde{d} \colon Y^2 \to \RR\) the metric identification of \(d\), where \(Y = X / \coherent{0}\). Then the semigroups \(\mathcal{B}_X(P_{d^{-1}})\) and \(\mathcal{B}_Y(P_{\tilde{d}^{-1}})\) are isomorphic.
\end{proposition}

\begin{proof}
Let \(P = \{X_j \colon j \in J\}\) be the partition of \(X\) corresponding to the equivalence relation \(\coherent{0}\) generated by the pseudometric \(d\). Using Example~\ref{ch2:ex5} and Theorem~\ref{p2.7}, for every \(r \in d(X^{2})\), we obtain the equality
\begin{equation}\label{p4.16:e1}
d^{-1}(r) = \bigcup_{\<x, y> \in d^{-1}(r)} X_{j_x} \times X_{j_y},
\end{equation}
where \(X_{j_x}\) and  \(X_{j_y}\) are the blocks of \(P\) such that \(x \in X_{j_x}\) and \(y \in X_{j_y}\).

By Proposition~\ref{ch2:p1}, the elements of the Cartesian square \(Y^{2}\) are the ordered pairs \(\<X_i, X_j>\),
\[
Y^{2} = \{\<X_i, X_j> \colon X_i, X_j \in P\}.
\]
Let us consider a mapping \(\Phi \colon \mathcal{B}_X \to \mathcal{B}_Y\) such that, for every \(A \subseteq X^{2}\),
\begin{equation}\label{p4.16:e2}
(\<X_i, X_j> \in \Phi(A)) \Leftrightarrow (\exists \<x,y> \in A \text{ with } x \in X_i \text{ and } y \in X_j).
\end{equation}
From~\eqref{p4.16:e1} and \eqref{p4.16:e2} it follows that
\begin{equation}\label{p4.16:e3}
\Phi(P_{d^{-1}}) = P_{\tilde{d}^{-1}}
\end{equation}
and
\begin{equation}\label{p4.16:e4}
\bigl(\Phi(A) = \Phi(B)\bigr) \Leftrightarrow \left(\bigcup_{\<x, y> \in A} X_{j_x} \times X_{j_y} = \bigcup_{\<x, y> \in B} X_{j_x} \times X_{j_y}\right)
\end{equation}
for all \(A\), \(B \in \mathcal{B}_X(P_{d^{-1}})\), where \(X_{j_x}\) and \(X_{j_y}\) are defined as in~\eqref{p4.16:e1}.

Let \(\<X_i, X_j> \in \Phi(d^{-1}(r) \circ d^{-1}(s))\). Then there is \(\<x, y> \in d^{-1}(r) \circ d^{-1}(s)\) with \(x \in X_i\) and \(y \in X_j\) for some \(i\), \(j \in J\). Consequently, there is \(z \in X\) such that
\begin{equation}\label{p4.16:e5}
\<x, z> \in d^{-1}(r) \text{ and } \<z, y> \in d^{-1}(s).
\end{equation}
Let \(k \in J\) be such that \(z \in X_k\). Using Proposition~\ref{ch2:p1} we see that \eqref{p4.16:e5} holds if and only if
\[
X_i \times X_k \subseteq d^{-1}(r) \text{ and } X_k \times X_j  \subseteq d^{-1}(s).
\]
Now from~\eqref{p4.16:e2} it follows that
\[
\<X_i, X_k> \in \Phi(d^{-1}(r)) \text{ and } \<X_k, X_j> \in \Phi(d^{-1}(s)).
\]
Thus, we have
\[
\<X_i, X_j> \in \Phi(d^{-1}(r)) \circ \Phi(d^{-1}(s)).
\]
It implies the inclusion
\[
\Phi(d^{-1}(r) \circ d^{-1}(s)) \subseteq \Phi(d^{-1}(r)) \circ \Phi(d^{-1}(s)).
\]
The converse inclusion can be proved similarly. The equality
\begin{equation}\label{p4.16:e6}
\Phi(d^{-1}(r) \circ d^{-1}(s)) = \Phi(d^{-1}(r)) \circ \Phi(d^{-1}(s))
\end{equation}
follows. Since \(P_{d^{-1}}\) and \(P_{\tilde{d}^{-1}}\), respectively, are the set of generators of \(\mathcal{B}_X(P_{d^{-1}})\) and \(\mathcal{B}_Y(P_{\tilde{d}^{-1}})\), respectively, from~\eqref{p4.16:e3}, \eqref{p4.16:e4} and \eqref{p4.16:e6} it follows that the mapping \(\Phi|_{\mathcal{B}_X (P_{d^{-1}})}\) is an isomorphism of \(\mathcal{B}_X(P_{d^{-1}})\) and \(\mathcal{B}_Y(P_{\tilde{d}^{-1}})\).
\end{proof}

Proposition~\ref{p4.16} shows, in particular, that the converse to Proposition~\ref{p4.14} is false in general.

\begin{example}\label{ex4.17}
Let \((X, d)\) and \((Y, \rho)\) be nonempty pseudometric spaces with the zero pseudometrics and let \(|X| \neq |Y|\). Then the semigroups \(\mathcal{B}_X(P_{d^{-1}})\) and \(\mathcal{B}_Y (P_{\rho^{-1}})\) are isomorphic to the trivial group (consisting of a single element) but \(d\) and \(\rho\) are not combinatorially similar.
\end{example}

Our next goal is to describe up to isomorphism the algebraic structure of the semigroups \(\mathcal{B}_X(P_{d^{-1}})\) for strongly rigid pseudometrics~\(d\). To do this, we recall some concepts of semigroup theory.

In what follows we say that a nonempty subset \(\mathcal{H}_0\) of a semigroup \((\mathcal{H}, *)\), \(\mathcal{H}_0 \subseteq \mathcal{H}\), is a \emph{subsemigroup} of \((\mathcal{H}, *)\) if \(\mathcal{H}_0 * \mathcal{H}_0 \subseteq \mathcal{H}_0\). Moreover, a subsemigroup \(\mathcal{H}_1\) of a semigroup \((\mathcal{H}, *)\) is an \emph{ideal} of \(\mathcal{H}\) if 
\[
\mathcal{H}_1 * \mathcal{H} \subseteq \mathcal{H}_1 \quad \text{and} \quad \mathcal{H} * \mathcal{H}_1 \subseteq \mathcal{H}_1
\]
holds. 

\begin{remark}\label{r4.21}
We write
\begin{equation}\label{e4.10}
A*B := \{x*y \colon x \in A,\ y \in B\}
\end{equation}
for all nonempty subsets \(A\) and \(B\) of \((\mathcal{H}, *)\).
\end{remark}

Let \((\mathcal{S}, *)\) be an arbitrary semigroup and let \(\{e\}\) be a single-point set such that \(e \notin \mathcal{S}\). We can extend the multiplication \(*\) from \(\mathcal{S}\) to \(\mathcal{S} \cup \{e\}\) by the following rule:
\begin{equation}\label{e4.11}
e*e = e \quad \text{and} \quad e * x = x *e = x
\end{equation}
for every \(x \in \mathcal{S}\). Following~\cite{Clifford1961} we use the notation
\begin{equation}\label{e4.12}
\mathcal{S}^{1} := \begin{cases}
\mathcal{S}, & \text{if \((\mathcal{S}, *)\) has an identity element}\\
\mathcal{S} \cup \{e\}, & \text{otherwise}.
\end{cases}
\end{equation}
It is clear that \(e\) is an identity element of \((\mathcal{S}^{1}, *)\). Thus the semigroup \((\mathcal{S}^{1}, *)\) is obtained from \((\mathcal{S}, *)\) by ``adjunction of an identity element to \((\mathcal{S}, *)\)''.

Let \((\mathcal{S}, *)\) be a semigroup. If \(\mathcal{S}\) is a single-point set, \(\mathcal{S} = \{e\}\), then we conclude that \(e\) is the identity element of \((\mathcal{S}, *)\). Let \(|\mathcal{S}| \geqslant 2\). Then we call \(\theta \in \mathcal{S}\) a \emph{zero} element or simply zero of \((\mathcal{S}, *)\) if
\[
\theta * s = s * \theta = \theta
\]
holds for every \(s \in \mathcal{S}\). A zero element, if it exists, is unique.

An element \(i \in \mathcal{S}\) is an \emph{idempotent element} of \((\mathcal{S}, *)\) if
\[
i^2 = i * i = i.
\]
It is clear that the identity element \(e\) and the zero \(\theta\) are idempotents. We will say that \(e\) and \(\theta\) are the \emph{trivial idempotent elements}. In the sequel ``a nontrivial idempotent'' means an idempotent which is not the zero element if such exists and which is not the identity element if such exists. A semigroup is a \emph{band} if every element of this semigroup is idempotent.

\begin{definition}[{\cite{Dovgoshey2019}}]\label{d4.14}
Let \((\mathcal{H}, *)\) be a semigroup and let \(\mathcal{C}\) be an ideal of \((\mathcal{H}, *)\). The semigroup \((\mathcal{H}, *)\) is a \emph{band of subsemigroups with core} \(\mathcal{C}\) if \(\mathcal{C} \neq \mathcal{H}\) and there is a partition \(\{\mathcal{H}_{\alpha} \colon \alpha \in \Omega\}\) of the set \(\mathcal{H} \setminus \mathcal{C}\) such that every \(\mathcal{H}_{\alpha}\) is a subsemigroup of \(\mathcal{H}\) and \(\mathcal{H}_{\alpha_1} * \mathcal{H}_{\alpha_2} \subseteq \mathcal{C}\) holds for all distinct \(\alpha_1\), \(\alpha_2 \in \Omega\).
\end{definition}

If \(\mathcal{H}\) is a band of subsemigroups with core \(\mathcal{C}\), then we write
\[
\mathcal{H} \approx \{\mathcal{H}_\alpha \colon \alpha \in \Omega\} \sqcup \mathcal{C}.
\]

Recall that a semigroup \((\mathcal{H}, *)\) is a \emph{union of band of subsemigroups} \(\mathcal{H}_{\alpha}\), \(\alpha \in \Omega\), of \((\mathcal{H}, *)\) if \(\{\mathcal{H}_{\alpha} \colon \alpha \in \Omega\}\) is a partition of \(\mathcal{H}\) and, moreover, for every pair of distinct \(\alpha\), \(\beta \in \Omega\) there is \(\gamma \in \Omega\) such that \(\mathcal{H}_{\alpha} * \mathcal{H}_{\beta} \subseteq \mathcal{H}_{\gamma}\) (see, for example, \cite[p.~25]{Clifford1961}). The above defined band of subsemigroups with given core can be considered as a special case of the union of band of subsemigroups. Indeed, if \(\mathcal{H} \approx \{\mathcal{H}_\alpha \colon \alpha \in \Omega\} \sqcup \mathcal{C}\), then \(\mathcal{H}\) is the union of band of subsemigroups \(\mathcal{H}_\alpha\), \(\alpha \in \Omega\), and \(\mathcal{C}\).

Let us denote by \(\mathbf{H}_1\) the class of all semigroups \((\mathcal{H}, *)\) such that either \(|\mathcal{H}| = 1\) or \((\mathcal{H}, *)\) satisfies the following conditions:
\begin{enumerate}
\item[\((1)\)] \(\mathcal{H}\) contains a zero element \(\theta\).
\item[\((2)\)] The equality 
\begin{equation*}
x * y = \theta
\end{equation*}
holds for all distinct idempotent elements \(x\), \(y \in \mathcal{H}\).
\item[\((3)\)] If \(i_{l}\) and \(i_{r}\) are nontrivial idempotent elements of \(\mathcal{H}\), then there is a unique nonzero \(a \in \mathcal{H}\) such that
\begin{equation*}
a = i_{l} * a * i_{r}.
\end{equation*}
\item[\((4)\)] For every nonzero \(a \in \mathcal{H}\) there is a unique pair \((i_{la}, i_{ra})\) of nontrivial idempotent elements of \(\mathcal{H}\) such that
\begin{equation*}
a = i_{la} * a * i_{ra}.
\end{equation*}
\item[\((5)\)] For all nonzero \(a\), \(b \in \mathcal{H}\) and for all pairs \((i_{la}, i_{ra})\) and \((i_{lb}, i_{rb})\) of nontrivial idempotent elements of \(\mathcal{H}\) such that \(a = i_{la} * a * i_{ra}\) and \(b = i_{lb} * b * i_{rb}\) and \(i_{ra} = i_{lb}\) it holds that \(a * b \neq \theta\).
\end{enumerate}

\begin{remark}\label{r4.22}
By Lemma~3.13 of \cite{Dovgoshey2019} a semigroup \((\mathcal{H}, *) \in \mathbf{H}_1\) contains an identity element if and only if \(|\mathcal{H}| = 1\). Thus, every nonzero, idempotent element is nontrivial if \(|\mathcal{H}| \geqslant 2\), although this result is not needed in the present paper.
\end{remark}

\begin{example}\label{ex4.21}
A semigroup \((\mathcal{H}, *)\) with \(|\mathcal{H}| = 2\) having a zero element \(\theta\) and an identity element \(e\) (e.g., \((\mathcal{H}, *) = (\mathbb{Z}_2, \cdot)\)) satisfies conditions \((1)\)--\((3)\) and \((5)\) but not condition \((4)\) although for every nonzero \(a \in \mathcal{H}\) there is a unique pair \((i_{la}, i_{ra})\) of \emph{nonzero} idempotents such that \(a = i_{la}* a * i_{ra}\).
\end{example}

\begin{theorem}\label{t4.15}
Let \((\mathcal{H}, *)\) be a semigroup. Then the following conditions~\ref{t4.15:s1} and \ref{t4.15:s2} are equivalent:
\begin{enumerate}
\item\label{t4.15:s1} There are a nonempty set \(X\) and a nondiscrete, strongly rigid pseudometric \(d \colon X^{2} \to \RR\) such that the semigroup \((\mathcal{H}, *)\) is isomorphic to \(\mathcal{B}_X(P_{d^{-1}})\).
\item\label{t4.15:s2} The semigroup \((\mathcal{H}, *)\) is obtained by adjunction of an identity element to a semigroup \(\widehat{\mathcal{H}}\) for which the following statements hold:
\begin{enumerate}
\item [\((ii_1)\)] \(\widehat{\mathcal{H}}\) is a band of subsemigroups with core \(\mathcal{C}\),
\[
\widehat{\mathcal{H}} \approx \{\mathcal{H}_\alpha \colon \alpha \in \Omega\} \sqcup \mathcal{C},
\]
such that \(\mathcal{C} \in \mathbf{H}_{1}\) and \(\mathcal{H}_{\alpha}\) is a group of order \(2\) for every \(\alpha \in \Omega\).
\item [\((ii_2)\)] The double inequality \(3 \leqslant |\Omega| \leqslant \mathfrak{c}\) holds.
\item [\((ii_3)\)] The set \(\widehat{E}\) of all idempotents of \(\widehat{\mathcal{H}}\) is a commutative band.
\item [\((ii_4)\)] If \(e_1\), \(e_2\) are two distinct nontrivial idempotents of \(\mathcal{C}\), then there is a unique \(e \in \widehat{E}\setminus \mathcal{C}\) such that
\begin{equation}\label{t4.15:e1}
e_1 = e_1*e \quad \text{and} \quad e_2 = e_2*e.
\end{equation}
Conversely, if \(e \in \widehat{E} \setminus \mathcal{C}\), then there are exactly two distinct nontrivial \(e_1\), \(e_2 \in \mathcal{C} \cap \widehat{E}\) such that~\eqref{t4.15:e1} holds.
\item [\((ii_5)\)] For every \(x \in \widehat{E} \cap \mathcal{C}\) and every \(y \in \widehat{\mathcal{H}} \setminus \widehat{E}\) the equality \(x*y=\theta\) (\(y*x = \theta\)) holds if and only if \(x*y\) (\(y*x\)) is idempotent.
\end{enumerate}
\end{enumerate}
\end{theorem}

Theorem~\ref{t4.15} follows from Theorem~\ref{p4.6} of the present paper and Theorems~4.6 and 4.9 of \cite{Dovgoshey2019}. The following remark eliminates a gap in the proof of Theorem~4.6 of \cite{Dovgoshey2019}.

\begin{remark}\label{r5.3}
To get the intended proof for \(x_3*x=\theta\) in \((4.49)\) in the proof of \cite[Theorem 4.6]{Dovgoshey2019}, show \(x_1\ne x_3*x\ne x_2\) by contradiction. Alternatively, since \(x_3\in E\cap \mathcal{C}\), \(x\in E\setminus \mathcal{C}\) and \(j_3\notin\{j_1,j_2\}\) and since \((4.30)\) holds in case \((4.32)\),  get directly
\[
\Phi_1(x_3*x) = \Phi(x_3)\circ\Phi_1(x) = X_{j_3}^2\circ (X_{j_1}^2\cup X_{j_2}^2) = \varnothing = \Phi_1(\theta)
\]
and apply the injectivity of \(\Phi_1\).
\end{remark}

For Theorem~3.2 of \cite{Dovgoshey2019}, used in the proof of Theorem~4.6 of \cite{Dovgoshey2019}, and its proof, see Theorem~\ref{ch2:th6} below.

We do not discuss the details of the proof of Theorem~\ref{t4.15} but it should be noted that:
\begin{itemize}
\item Theorem~4.9 \cite{Dovgoshey2019} describes the algebraic structure of the subsemigroups of \(\mathcal{B}_X\) generated by partitions \(P \otimes P_S^1\) of \(X^{2}\) with \(P\) a partition of \(X\) for the case when \(|P| \geqslant 2\). There the condition \(|P| \geqslant 2\) must also be replaced by the condition \(|P| \geqslant 3\).
\item To apply Theorem~\ref{p4.6}, recall that if \(X\) is a nonempty set and \(Q\) is a partition of \(X^{2}\), then \(P_{\Phi^{-1}} = Q\) for the surjection \(\Phi \colon X^{2} \to Q\) with \(\Phi(x, y) = A\) whenever \(\<x, y> \in A \in Q\).
\item If \(d \colon X^{2} \to \RR\) is a nondiscrete and strongly rigid pseudometric, then we have
\[
P_{d^{-1}} = P \otimes P_S^1
\]
with a partition \(P\) of \(X\) such that \(|P| \geqslant 3\).
\end{itemize}

The next proposition clarifies this situation.

Recall that a semigroup is said to be a \emph{null} semigroup if a product of every two elements is zero \cite[p.~3]{Howie1976}.

\begin{proposition}\label{p4.22}
The following conditions are equivalent for every semigroup \((\mathcal{H}, *)\):
\begin{enumerate}
\item \label{p4.22:c1} There are a nonempty set \(X\) and a discrete pseudometric \(d \colon X^{2} \to \RR\) such that \((\mathcal{H}, *)\) is isomorphic to \(\mathcal{B}_X (P_{d^{-1}})\).
\item \label{p4.22:c2} Exactly one of the following three conditions holds:
\begin{enumerate}
\item \((\mathcal{H}, *)\) is a trivial group.
\item \((\mathcal{H}, *)\) is a group of order \(2\).
\item \((\mathcal{H}, *)\) is obtained by adjunction of an identity element to a null semigroup of order \(2\).
\end{enumerate}
\end{enumerate}
\end{proposition}

\begin{proof}
\(\ref{p4.22:c1} \Rightarrow \ref{p4.22:c2}\). Let \(d \colon X^{2} \to \RR\) be a discrete pseudometric for which the semigroups \((\mathcal{H}, *)\) and \(\mathcal{B}_X (P_{d^{-1}})\) are isomorphic. Since \(d\) is discrete, the inequality \(|d(X^{2})| \leqslant 2\) holds. Hence, we have \(|P_{d^{-1}}| \leqslant 2\). If \(|P_{d^{-1}}| = 1\) is valid, then \(d\) is the zero pseudometric on \(X\) and \(X^2 = \coherent{0}\) holds. Thus, \(P_{d^{-1}} = \{X^2\}\) holds, which implies \((ii_1)\).

Let us consider the case when \(|P_{d^{-1}}| = 2\). Then the relation \(d^{-1}(0) \in P_{d^{-1}}\) is an equivalence relation on \(X\) and \(d^{-1}(0) \neq X^2\). By Lemma~\ref{l4.5}, there is a partition \(P = \{X_j \colon j \in J\}\) of \(X\) such that
\[
|J| \geqslant 2 \quad \text{and} \quad d^{-1}(0) = \bigcup_{j \in J} X_j^{2}
\]
hold. If \(|J| = 2\), then we obtain (for \(J = \{1,2\}\))
\[
P_{d^{-1}} = \{X_1^2 \cup X_2^2, (X_1 \times X_2) \cup (X_2 \times X_1)\}.
\]
Write \(e := X_1^2 \cup X_2^2\) and \(e_1 := (X_1 \times X_2) \cup (X_2 \times X_1)\) for simplicity. A direct calculation shows that
\[
e_1 \circ e_1 = e \circ e = e \text{ and } e_1 \circ e = e \circ e_1 = e_1.
\]
Hence, \(\mathcal{B}_X(P_{d^{-1}})\) is a group of order \(2\). Since \((\mathcal{H}, *)\) and \(\mathcal{B}_X(P_{d^{-1}})\) are isomorphic, \((ii_2)\) holds.

Suppose now \(|J| \geqslant 3\). Then we write \(P_{d^{-1}} = \{e, e_0\}\), where 
\[
e := \bigcup_{j \in J} X_j^{2} \quad \text{and} \quad e_0 := X^{2} \setminus \left(\bigcup_{j \in J} X_j^{2} \right).
\]
We claim that the equality
\begin{equation}\label{p4.22:e1}
e_0^2 = X^2
\end{equation}
holds. Indeed, it is clear that \(e_0 \circ e_0 \subseteq X^2\), so it suffices to show that
\begin{equation}\label{p4.22:e2}
\<x, y> \in e_0^2
\end{equation}
holds for all \(x\), \(y \in X\). Let \(x\) and \(y\) be points of \(X\). Suppose \(\<x,y> \in e\). Then there is \(j_0 \in J\) such that
\[
x \in X_{j_0} \quad \text{and} \quad y \in X_{j_0}.
\]
Let \(j_1 \in J\) and \(j_1 \neq j_0\) and \(z \in X_{j_1}\). Then we obtain
\begin{equation}\label{p4.22:e3}
\<x, z> \in X_{j_0} \times X_{j_1} \subseteq e_0
\end{equation}
and
\begin{equation}\label{p4.22:e4}
\<z, y> \in X_{j_1} \times X_{j_0} \subseteq e_0.
\end{equation}
From~\eqref{p4.22:e3} and \eqref{p4.22:e4} it follows that
\begin{equation}\label{p4.22:e5}
\<x, y> \in e_0 \circ e_0 = e_0^2.
\end{equation}
Similarly, if \(\<x, y> \notin e\) holds, then there are distinct \(i_0\), \(i_1 \in J\) such that \(x \in X_{i_0}\), \(y \in X_{i_1}\) and \(i_2 \in J\) with
\[
i_0 \neq i_2 \neq i_1
\]
(recall that \(|J| \geqslant 3\)). Let \(z \in X_{i_2}\). Then
\[
\<x, z> \in X_{i_0} \times X_{i_2} \subseteq e_0 \quad \text{and} \quad \<z, y> \in X_{i_2} \times X_{i_1} \subseteq e_0
\]
are valid and, consequently, we have \eqref{p4.22:e5} again. Equality~\eqref{p4.22:e1} follows.

It is easy to prove that 
\[
e_0^2 \circ e_0 = e_0 \circ e_0^2 = e_0^2 \circ e = e \circ e_0^2 = e_0^2.
\]
Thus, the set \(\{e_0, e_0^2\}\) with the multiplication \(\circ\) is a null semigroup of order \(2\) and the semigroup \(\{e_0, e_0^2, e\}\) is obtained by adjunction of the identity element \(e\) to this semigroup.

\(\ref{p4.22:c2} \Rightarrow \ref{p4.22:c1}\). Suppose \(\ref{p4.22:c2}\) holds. Let \(X\) be a set with \(|X| \geqslant 3\). To prove \(\ref{p4.22:c1}\) it suffices to note that each of conditions \((ii_1)\), \((ii_2)\) and \((ii_3)\) describes \((\mathcal{H}, *)\) up to isomorphism and to consider discrete pseudometrics \(d \colon X^{2} \to \RR\) for which the equivalence relation \(\coherent{0}\) generates \(1\), or \(2\), or \(3\) distinct equivalence classes. The existence of such pseudometrics, which is obvious in itself, is also guaranteed by Corollary~\ref{ch2:p2}. 
\end{proof}

The following theorem shows that the semigroups belonging to the class \(\mathbf{H}_1\) (see Theorem~\ref{t4.15}) coincide, up to isomorphism, with the semigroups generated by all one-point subsets of the Cartesian square \(X^{2}\) with \(X \neq \varnothing\).

\begin{theorem}\label{ch2:th6}
Let \((\mathcal{H}, *)\) be a semigroup. Then the following two statements are equivalent:
\begin{enumerate}
\item\label{ch2:th6:s1} There are a nonempty set \(X\) and a partition \(P = \{X_j \colon j \in J\}\) of \(X\) such that \((\mathcal{H}, *)\) is isomorphic to \(\mathcal{B}_X(P\otimes P)\), where \(P \otimes P = \{X_i \times X_j \colon i, j \in J\}\).
\item\label{ch2:th6:s2} \((\mathcal{H}, *) \in \mathbf{H}_1\).
\end{enumerate}
\end{theorem}

\begin{proof}
In what follows we use the notation \(\mathcal{S}_{P\otimes P}\) as in~\cite{Dovgoshey2019} in place of \(\mathcal{B}_X(P\otimes P)\).

\(\ref{ch2:th6:s1} \Rightarrow \ref{ch2:th6:s2}\). Let \(P = \{X_j \colon j \in J\}\) be a partition of a nonempty set \(X\) such that the semigroup \((\mathcal{H}, *)\) is isomorphic to \((\mathcal{S}_{P\otimes P}, \circ)\). If \(|P| = 1\), then \(|\mathcal{H}| = 1\) implying \((\mathcal{H}, *) \in \mathbf{H}_1\). Suppose \(|P| \geqslant 2\). It was shown in the proof of Theorem~3.2 \cite{Dovgoshey2019} that now conditions \((1)\)--\((4)\) from the definition of \(\mathbf{H}_1\) are satisfied; here for \((4)\) note that it follows from formula \cite[(3.4)]{Dovgoshey2019}, i.e., \eqref{ch2:th6:e4} below, that \(\mathcal{S}_{P\otimes P}\) does not contain an identity element. We must only prove that \((\mathcal{H}, *)\) satisfies condition~\((5)\). We may assume that \((\mathcal{H}, *) = (\mathcal{S}_{P\otimes P}, \circ)\).

Let \(a\) and \(b\) be elements of \(\mathcal{S}_{P \otimes P}\) such that \(a \neq \varnothing \neq b\). If \((i_{la}, i_{ra})\) and \((i_{lb}, i_{rb})\) are the pairs of nontrivial idempotent elements of \(\mathcal{S}_{P \otimes P}\) for which
\[
a = i_{la} \circ a \circ i_{ra} \quad \text{and} \quad b = i_{lb} \circ b \circ i_{rb}
\]
and if \(i_{ra} = i_{lb}\), then using the equality
\begin{equation}\label{ch2:th6:e4}
(X_{j_1} \times X_{j_2}) \circ (X_{j_3} \times X_{j_4}) = \begin{cases}
\varnothing, & \text{if } X_{j_2} \neq X_{j_3}\\
X_{j_1} \times X_{j_4}, & \text{if } X_{j_2} = X_{j_3}
\end{cases}
\end{equation}
we can find \(i\), \(j\), \(k \in J\) such that
\[
a = X_i \times X_j \quad \text{and} \quad b = X_j \times X_k
\]
implying
\[
a \circ b = X_i \times X_k \neq \varnothing.
\]

\(\ref{ch2:th6:s2} \Rightarrow \ref{ch2:th6:s1}\). Let \((\mathcal{H}, *)\) satisfy condition~\ref{ch2:th6:s2}. Condition \ref{ch2:th6:s1} is trivially valid for \(|\mathcal{H}| = 1\). Suppose that \(|\mathcal{H}| \geqslant 2\). Let \(E\) be the set of all nontrivial idempotent elements of \(\mathcal{H}\) and let \(P\) be the partition of \(E\) into the one-point subsets of \(E\),
\[
P = \{\{i\} \colon i \in E\}.
\]

We claim that the semigroup \((\mathcal{S}_{P\otimes P}, \circ)\) is isomorphic to \((\mathcal{H}, *)\). 

By formula~\eqref{ch2:th6:e4} we see that every element of \((\mathcal{S}_{P\otimes P}, \circ)\) is either empty or has the form \(s = \{i_1\} \times \{i_2\}\) for some \(i_1\), \(i_2 \in E\). From the definition of the Cartesian product, we have the equality 
\[
\{i_1\} \times \{i_2\} = \{\<i_1, i_2>\};
\]
thus
\begin{equation}\label{ch2:th6:e7}
s = \{\<i_1, i_2>\},
\end{equation}
where \(\<i_1, i_2> \in E \times E\). Conditions \((3)\) and \((4)\) imply that there is a bijection \(F \colon \mathcal{S}_{P\otimes P} \to \mathcal{H}\) such that \(F(\varnothing) = \theta\), where \(\theta\) is the zero element of \((\mathcal{H}, *)\), and \(F(\{\<i_1, i_2>\}) = x\), where \(x\) is the unique nonzero element of \(\mathcal{H}\) such that 
\begin{equation*}
x = i_1 * x * i_2.
\end{equation*}
It suffices to show that \(F \colon \mathcal{S}_{P\otimes P} \to \mathcal{H}\) is an isomorphism. Let \(s_1\) and \(s_2\) belong to \(\mathcal{S}_{P\otimes P}\). We must show that 
\begin{equation*}
F(s_1) * F(s_2) = F(s_1 \circ s_2).
\end{equation*}
This equality is trivially valid if \(s_1 = \varnothing\) or \(s_2 = \varnothing\). Suppose now that \(s_1 \neq \varnothing \neq s_2\) but 
\begin{equation}\label{ch2:th6:e11}
s_1 \circ s_2 = \varnothing
\end{equation}
holds. Using~\eqref{ch2:th6:e7} we can find \(i_{1,1}\), \(i_{1,2}\), \(i_{2,1}\), \(i_{2,2} \in E\) such that
\begin{equation}\label{ch2:th6:e12}
s_1 = \{\<i_{1,1}, i_{1,2}>\} \quad \text{and} \quad s_2 = \{\<i_{2,1}, i_{2,2}>\}.
\end{equation}
From~\eqref{ch2:th6:e11} it follows that \(i_{1,2} \neq i_{2,1}\). By \((3)\), there are unique nonzero \(x_1\), \(x_2 \in \mathcal{H}\) such that
\begin{equation}\label{ch2:th6:e13}
x_1 = i_{1,1} * x_1 * i_{1,2}\quad \text{and} \quad x_2 = i_{2,1} * x_2 * i_{2,2}.
\end{equation}
By definition of \(F\), we have the equalities \(F(\varnothing) = \theta\) and \(F(s_i) = x_i\) for \(i = 1\), \(2\). Now using~\eqref{ch2:th6:e13}, condition \((2)\) and \(i_{1,2} \neq i_{2,1}\) we obtain
\begin{multline*}
F(s_1)*F(s_2) = x_1 * x_2 = i_{1,1} * x_1 * (i_{1,2} * i_{2,1}) * x_2 * i_{2,2} \\
= i_{1,1} * x_1 * \theta * x_2 * i_{2,2} = \theta = F(\varnothing).
\end{multline*}

Suppose now that \(s_1\), \(s_2\), \(s_1 \circ s_2\) are nonzero elements of \((\mathcal{S}_{P\times P}, \circ)\). Then \eqref{ch2:th6:e12} and \eqref{ch2:th6:e13} hold with \(i_{1,2} = i_{2,1}\) and \(s_1 \circ s_2 = \{\<i_{1,1}, i_{2,2}>\}\). Hence, we get \(x_1 * x_2 \neq \theta\) by \((5)\) and
\begin{equation}\label{ch2:th6:e14}
F(s_1)*F(s_2) = (i_{1,1} * x_1 * i_{1,2}) * (i_{2,1} * x_2 * i_{2,2}) = x_1 * x_2.
\end{equation}
Since \(i_{1,1}\), \(i_{1,2}\), \(i_{2,2}\) are idempotent, from~\eqref{ch2:th6:e13} we have
\[
i_{1,1} * x_1 * i_{1,2} = i_{1,1} * x_1 \quad \text{and} \quad i_{2,1} * x_2 * i_{2,2} = x_2 * i_{2,2}.
\]
Now using \eqref{ch2:th6:e14} we obtain
\begin{equation*}
F(s_1)*F(s_2) = i_{1,1} * (x_1 * x_2) * i_{2,2}.
\end{equation*}
By definition of \(F\), there is a unique nonzero \(y \in \mathcal{H}\) such that
\[
F(\{\<i_{1,1}, i_{2,2}>\}) = y = i_{1,1} * y * i_{2,2}.
\]
Thus
\[
F(s_1 \circ s_2) = F(s_1)*F(s_2)
\]
holds for all \(s_1\), \(s_2 \in \mathcal{S}_{P\otimes P}\). The implication \(\ref{ch2:th6:s2} \Rightarrow \ref{ch2:th6:s1}\) follows.
\end{proof}

The following example shows that condition \((5)\) is independent of conditions \((1)\)--\((4)\) in the definition of \(\mathbf{H}_1\).

\begin{example}\label{ex5.2}
Let \(X\) be a nonempty set and \(\{X_j \colon j \in J\}\) a partition of \(X\) with \(|J| \geqslant 2\). Write
\[
\mathcal{H} := \{X_i \times X_j \colon i, j \in J\} \cup \{\varnothing\},
\]
and define a binary operation \(*\) on \(\mathcal{H}\) such that \(x*y = \varnothing\) if \(\<x, y> = \<X_i \times X_j, X_j \times X_k>\) for some \(i\), \(j\), \(k \in J\) with \(i \neq j \neq k\) and such that \(x * y = x \circ y\) otherwise. Then, for all \(x\), \(y\), \(z \in \mathcal{H}\), we have 
\[
x * (y * z) = \varnothing \Leftrightarrow (x * y) * z = \varnothing
\]
and otherwise
\[
x * (y * z) = x \circ (y \circ z) = (x \circ y) \circ z = (x * y) * z.
\]
Thus, the operation \(*\) is associative with a zero element \(\varnothing\), and the semigroup \((\mathcal{H}, *)\) satisfies conditions \((1)\)--\((4)\), but \((\mathcal{H}, *)\) does not satisfy condition \((5)\).
\end{example}

The initial variant of the definition of \(\mathbf{H}_1\) \cite[Theorem~3.2]{Dovgoshey2019} contains a gap. In particular, condition \((5)\) was absent there.


\begin{thebibliography}{10}

\bibitem{And1976}
G.~E. Andrews, \emph{{The Theory of Partitions}}, Encyclopedia of Mathematics
  and its Applications, vol.~2, Cambridge University Press, Cambridge, 1998.

\bibitem{Blum(1953)}
L.~M. Blumenthal, \emph{{Theory and Applications of Distance Geometry}},
  Clarendon Press, Oxford, 1953.

\bibitem{Clifford1961}
A.~H. Clifford and G.~B. Preston, \emph{{The Algebraic Theory of Semigroups.
  Part 1}}, Mathematical Surveys and Monographs, vol.~7, AMS, Providence, RI, 1961.

\bibitem{Dovgoshey2019}
O.~Dovgoshey, \emph{{Semigroups generated by partitions}}, Int. Electron. J.
  Algebra \textbf{26} (2019), 145--190.

\bibitem{DP}
O.~Dovgoshey and E.~Petrov, \emph{{Weak similarities of metric and semimetric
  spaces}}, Acta Math. Hung. \textbf{141} (2013), no.~4, 301--319.

\bibitem{Howie1976}
J.~M. Howie, \emph{An Introduction to Semigroup Theory}, Academic Press, London
  --- New York --- San Francisco, 1976.

\bibitem{Janos1972}
L.~Janos, \emph{A metric characterization of zero-dimensional spaces}, Proc.
  Amer. Math. Soc. \textbf{31} (1972), no.~1, 268--270.

\bibitem{Kelley1965}
J.~L. Kelley, \emph{{General Topology}}, Springer-Verlag, New York --- Heidelberg --- Berlin, 1975.

\bibitem{KurMost}
K.~Kuratowski and A.~Mostowski, \emph{{Set Theory with an Introduction to
  Descriptive Set Theory}}, North-Holland Publishing Company, Amsterdam --- New
  York --- Oxford, 1976.

\bibitem{Lang2005}
S.~Lang, \emph{{Undergraduate Algebra}}, {Third} ed., Springer, New York, 2005.

\bibitem{Martin1977}
H.~W. Martin, \emph{Strongly rigid metrics and zero dimensionality}, Proc.
  Amer. Math. Soc. \textbf{67} (1977), no.~1, 157--161.

\bibitem{Schoenberg1940}
I.~J. Schoenberg, \emph{{On metric arcs of vanishing Menger curvature}}, Ann.
  of Math. (2) \textbf{41} (1940), no.~4, 715--726.

\bibitem{Schoenberg1952}
I.~J. Schoenberg, \emph{{A remark on M. M. Day's characterization of inner-product
  spaces and a conjecture of L. M. Blumenthal}}, Proc. Amer. Math. Soc. \textbf{3}
  (1952), no.~6, 961--964.

\end{thebibliography}

\providecommand{\bysame}{\leavevmode\hbox to3em{\hrulefill}\thinspace}
\providecommand{\MR}{\relax\ifhmode\unskip\space\fi MR }
\providecommand{\MRhref}[2]{%
  \href{http://www.ams.org/mathscinet-getitem?mr=#1}{#2}
}
\providecommand{\href}[2]{#2}

\end{document}